\documentclass[12pt]{amsart}
\usepackage{a4wide}
\usepackage{amssymb}
\usepackage{amsthm}
\usepackage{amsmath}
\usepackage{amscd}
\usepackage{verbatim}
\usepackage[all]{xy}
\usepackage{longtable, lscape}

\textheight8.5in \textwidth6.4in \numberwithin{equation}{section}

\theoremstyle{plain}
\newtheorem{theorem}{Theorem}[section]
\newtheorem{corollary}[theorem]{Corollary}
\newtheorem{lemma}[theorem]{Lemma}
\newtheorem{proposition}[theorem]{Proposition}

\newtheorem{question}[theorem]{Question}

\theoremstyle{definition}
\newtheorem{definition}[theorem]{Definition}
\newtheorem{example}[theorem]{Example}

\theoremstyle{remark}
\newtheorem{remark}[theorem]{Remark}

\newcommand{\R}{\mathbb{R}}
\newcommand{\Q}{\mathbb{Q}}
\newcommand{\Z}{\mathbb{Z}}

\newcommand{\C}{\mathbb{C}}

\renewcommand{\H}{\mathbb{H}}

\newcommand{\F}{\mathbb{F}}
\newcommand{\D}{\mathbb{D}}

\newcommand{\zxz}[4]{\begin{pmatrix} #1 & #2 \\ #3 & #4 \end{pmatrix}}

\newcommand{\leg}[2]{\left( \frac{#1}{#2} \right)}
\newcommand{\kzxz}[4]{\left(\begin{smallmatrix} #1 & #2 \\ #3 & #4\end{smallmatrix}\right) }
\newcommand{\kabcd}{\kzxz{a}{b}{c}{d}}

\newcommand{\calH}{\mathcal{H}}

\newcommand{\calK}{\mathcal{K}}
\newcommand{\calL}{\mathcal{L}}

\newcommand{\calO}{\mathcal{O}}

\newcommand{\calV}{\mathcal{V}}

\newcommand{\frake}{\mathfrak e}

\newcommand{\eps}{\varepsilon}
\newcommand{\bs}{\backslash}

\newcommand{\Span}{\operatorname{span}}

\newcommand{\Sl}{\operatorname{SL}}
\newcommand{\Gl}{\operatorname{GL}}
\newcommand{\SL}{\operatorname{SL}}
\newcommand{\Symp}{\operatorname{Sp}}

\newcommand{\Mp}{\operatorname{Mp}}
\newcommand{\Orth}{\operatorname{O}}

\newcommand{\Hom}{\operatorname{Hom}}
\newcommand{\Aut}{\operatorname{Aut}}
\newcommand{\Mat}{\operatorname{Mat}}

\newcommand{\dv}{\operatorname{div}}

\newcommand{\sig}{\operatorname{sig}}

\newcommand{\Psp}{\operatorname{Pic}_{\text{\rm \scriptsize sp}}}

\newcommand{\Pic}{\operatorname{Pic}}

\newcommand{\Div}{\operatorname{Div}}

\newcommand{\cont}{\operatorname{cont}}
\newcommand{\ord}{\operatorname{ord}}
\newcommand{\rank}{\operatorname{rank}}

\begin{document}

\title[On the converse theorem for Borcherds products]{On the converse theorem for Borcherds products}

\author[Jan H.~Bruinier]{Jan
Hendrik Bruinier}
\dedicatory{To Eberhard Freitag}
\address{Fachbereich Mathematik,
Technische Universit\"at Darmstadt, Schlossgartenstrasse 7, D--64289
Darmstadt, Germany}
\email{bruinier@mathematik.tu-darmstadt.de}

\thanks{The author is partially supported by DFG grant BR-2163/2-2.}

\date{\today}

\begin{abstract}
  We prove a new converse theorem for Borcherds' multiplicative theta
  lift which improves the previously known results. To this end we
  develop a newform theory for vector valued modular forms for the
  Weil representation, which might be of independent interest.  We
  also derive lower bounds for the ranks of the Picard groups and the
  spaces of holomorphic top degree differential forms of modular
  varieties associated to orthogonal groups.
\end{abstract}

\maketitle

\section{Introduction}
\label{sect:intro}

In his celebrated paper \cite{Bo2} R.~Borcherds constructed a lift
from vector valued weakly holomorphic elliptic modular forms of weight
$1-n/2$ to meromorphic modular forms on the orthogonal group
$\Orth(n,2)$ whose zeros and poles are supported on special divisors
% (also called Heegner divisors)
and which possess infinite product expansions analogous to the Dedekind eta function.

Conversely, we prove in the present paper that in a large class of
cases every meromorphic modular form on $\Orth(n,2)$
whose divisor is supported on special divisors is the Borcherds lift
of a weakly holomorphic modular form of weight $1-n/2$.

%The main goal of the present paper is to prove in a large class of
%cases that conversely every meromorphic modular form on $\Orth(n,2)$
%whose divisor is supported on special divisors is the Borcherds lift
%of a weakly holomorphic modular form of weight $1-n/2$.

%It is a natural question to ask for a `converse theorem' for the lift:
%(see \cite{Bo1}, Problem 10 in Section 17 and \cite{Bo2} Problem 16.10):
%Is every meromorphic modular form on $\Orth(n,2)$ whose
%divisor is supported on special divisors the Borcherds lift of a
%weakly holomorphic modular form?  The main goal of the present paper
%is to answer this question in the affirmative in a large class of
%cases, improving the previously known results.
%To this end we develop
%a newform theory for vector valued modular forms, which might be of
%independent interest.
%As an application we obtain lower bounds for the ranks of
%the Picard groups of modular varieties associated to orthogonal
%groups.

%We now describe the results of the present paper in more detail.
Let $(V,Q)$ be a quadratic space over $\Q$ of signature $(n,2)$, and
let $\Orth(V)$ be its orthogonal group.  We realize the corresponding
hermitian symmetric space as the Grassmannian $\D$ of negative
definite oriented subspaces $z\subset V(\R)$ of dimension $2$.  
It has two connected components corresponding to the two possible
choices of an orientation.  
We fix one component and denote it by
$\D^+$.  The real orthogonal group $\Orth(V)(\R)$ acts transitively on
$\D$, and
%$\Orth(V)(\R)$ acts transitively on $\D$.
the subgroup $\Orth(V)(\R)^+$ of elements whose spinor norm has the
same sign as the determinant preserves $\D^+$.
%A subgroup $\Orth(V)(\R)^+$  of the real orthogonal group $\Orth(V)(\R)$ of index $2$
%(the subgroup of elements $\Orth(V)(\R)$ whose spinor norm has the same sign as the determinant)

Let $L\subset V$ be an even lattice, and let $L'$ be its dual. If $N$
is a non-zero integer, we write $L(N)$ for the lattice $L$ as a
$\Z$-module but equipped with the rescaled quadratic form
$NQ(\cdot)$. In addition, we briefly write $L^-=L(-1)$.
%We have $L(N)'=\frac{1}{N}L'$.
The quadratic form $Q$ on $L$ induces a $\Q/\Z$ valued quadratic form
on the discriminant group $L'/L$.
%We write $A^-$ for the discriminant form given by $A$ as an abelian group, but equipped with the quadratic form $-Q$.
We denote by $\Orth(L)$ the orthogonal group of $L$ and put $\Orth(L)^+=\Orth(L)\cap \Orth(V)(\R)^+$. % This is an arithmetic subgroup of $\Orth(V)(\R)$ which acts on $\D^+$ with finite covolume.
The kernel $\Gamma=\Gamma(L)$ of the natural map
$\Orth(L)^+\to\Aut(L'/L)$
is called the {\em discriminant kernel\/} subgroup of $\Orth(L)^+$.
%Throughout we let $\Gamma\subset \Orth(L)^+$ be a normal subgroup, which is contained in $\Gamma(L)$.
We consider the modular variety
\[
X_\Gamma = \Gamma\bs \D^+.
\]
By the theory of Baily-Borel, it carries the structure of a quasi-projective algebraic variety.
%----
For any $m\in \Q_{>0}$ and any $\mu\in L'/L$ there is a special
divisor $Z(m,\mu)$ on $X_\Gamma$ (also called Heegner divisor or
rational quadratic divisor), defined by the sum of the orthogonal
complements in $\D^+$ of vectors of norm $m$ in $L+\mu$, see Section
\ref{sect:4}.
%\eqref{eq:heeg}.

We realize the metaplectic group $\Mp_2(\Z)$ as the the non-trivial
twofold central extension of $\Sl_2(\Z)$ given by the two possible
choices of a holomorphic square root of the usual automorphy factor.
 %$c\tau+d$ for $\kabcd\in \Sl_2(\Z)$ and $\tau \in \H$.
Associated to the finite quadratic module $L'/L$, there exists a Weil
representation $\rho_L$ of $\Mp_2(\Z)$ on the group ring $\C[L'/L]$,
see Section \ref{sect:2.1}. The dual of $\rho_L$ can be identified
with $\rho_{L^-}$.  We denote by $M_{k,L^-}^!$ the space of  weakly
holomorphic modular forms for the group $\Mp_2(\Z)$ of weight $k$ and
representation $\rho_{L^-}$. Any $f\in M^!_{k,L^-}$ has a Fourier
expansion of the form
\begin{align*}
%\label{eq:few}
f(\tau)= \sum_{\mu\in L'/L}\sum_{m\in \Z-Q(\mu)}c(m,\mu)q^m\frake_\mu,
\end{align*}
where $q=e^{2\pi i\tau}$ for $\tau \in \H$, and $\frake_\mu$ denotes
the element of $\C[L'/L]$ given by the function $L'/L\to \C$ which is
$1$ on $\mu$ and $0$ for all $\nu\neq \mu$.  The main properties of
the Borcherds lift are summarized by the following theorem, see
Theorem 13.3 in \cite{Bo2}.

\begin{theorem}[Borcherds]
\label{thm:bo}
Let $f\in M_{1-n/2,L^-}^!$ be a weakly holomorphic modular form with Fourier coefficients $c(m,\mu)$ as above.
%in \eqref{eq:f.
% for $h\in A$ and $m\in \Z-Q(h)$
Assume that $c(m,\mu)\in \Z$ when $m<0$.
Then there exists a meromorphic modular form $\Psi(z,f)$ for the group $\Gamma$ with a unitary
multiplier system of finite order such that:
\begin{enumerate}
\item[(i)] The weight of $\Psi(z,f)$ is equal to $c(0,0)/2$.
\item[(ii)] The divisor of $\Psi(z,f)$ is given by
\[
Z(f)=\frac{1}{2}\sum_{\mu\in L'/L}\sum_{m>0} c(-m,\mu) Z(m,\mu).
\]
\item[(iii)]  $\Psi(z,f)$ has a particular infinite product expansion.
\end{enumerate}
\end{theorem}

%The lift is equivariant with respect to the actions of $\Orth(L)^+$ on
%$M_{1-n/2,A^-}^!$ and on meromorphic modular forms for
%$\Gamma$. Therefore, $\Psi(z,f)$ is actually modular for the
%stabilizer in $\Orth(L)^+$ of the weakly holomorphic modular form $f$.

In this paper we consider the question (asked by Borcherds in
\cite[Problem~10 in Section~17]{Bo1} and \cite[Problem~16.10]{Bo2})
whether there is a converse theorem for this result in the following
sense: Let $F$ be a meromorphic modular form for the group $\Gamma$
whose zeros and poles are supported on special divisors $Z(m,\mu)$ for
$\mu\in L'/L$ and $m\in \Q_{>0}$. Is there a weakly holomorphic form
$f\in M^!_{1-n/2,L^-}$ whose Borcherds lift $\Psi(z,f)$ is equal to
$F$ (up to a constant factor)?

\begin{comment}
\begin{question}
\label{qestion:1}
Let $F$ be a meromorphic modular
form for the group $\Gamma$ whose zeros and poles are supported on special divisors, that is,
\begin{align}
\label{eq:divf}
\dv(F)=\frac{1}{2}\sum_\mu\sum_{m>0} c(-m,\mu) Z(m,\mu).
\end{align}
%(where $c(n,h)=c(n,-h)$ without loss of generality).
Is there a weakly holomorphic form $f\in M_{1-n/2,A^-}$ whose Borcherds lift $\Psi(z,f)$
is equal to $F$ (up to a constant factor)?
\end{question}
\end{comment}

It is known that there are
counter examples for $n=1$, when $X_\Gamma$ is a curve.
For instance, if there is an elliptic curve $E/\Q$
of conductor $N$ whose $L$-function has an odd functional equation and
has order $\geq 3$ at the center $s=1$, then the Gross-Zagier formula
implies that the $E$-isotypical components of all Heegner divisors
$Z(m,\mu)$ on the modular curve $X_0(N)$ are torsion in the
Jacobian. Consequently, there are rational relations among the
$Z(m,\mu)$ which cannot be obtained as the Borcherds lift of a weakly
holomorphic modular form of weight $1/2$, see \cite[Section 8.3]{BO}.
If one considers a slight generalization of the above Heegner
divisors, allowing twists by genus characters of the corresponding CM
fields, then there are further counter examples related to Ramanujan's
mock theta functions, see \cite[Section 8.2]{BO}.

On the other hand, for large $n$ there are no known counter examples,
and there is the belief that a converse theorem might hold in this
case. Let $U$ be the two-dimensional even unimodular lattice of
signature $(1,1)$, realized as $\Z^2$ equipped with the quadratic form
$Q((x_1,x_2))=x_1x_2$.  Any lattice isomorphic to $U$ is called a
hyperbolic plane.  The best known result regarding the above question
states that a converse theorem holds if $L\cong D\oplus U\oplus U$ for
a positive definite even lattice $D$ of dimension $n-2$, see
\cite[Theorem~5.12]{Br1}.  This includes the special case when $L$ is
unimodular, for which alternative proofs are also given in
\cite{BrFr}, \cite{BF2}.  An analogous question for orthogonal groups
of signature $(n,1)$ was considered by Barnard in connection with
Lorentzian reflection groups \cite{Ba}.

In the present paper we prove the following stronger results:
% see Corollary \ref{thm:main}

\begin{theorem}
\label{thm:mainintro}
Assume that $L\cong D\oplus U(N)\oplus U$ for some positive definite even
lattice $D$ of dimension $n-2\geq 1$ and some positive integer $N$.
Then every meromorphic modular form $F$ with
respect to $\Gamma(L)$ whose divisor is a linear combination of special
divisors $Z(m,\mu)$
%as in \eqref{eq:divf}
%($c(\lambda)\in \Z$ with $c(\lambda)=c(-\lambda)$ and $c(\lambda)=0$ for all but finitely many $\lambda$),
is  (up to a non-zero constant factor)
the Borcherds lift $\Psi(z,f)$ of a weakly
holomorphic modular form $f\in M_{1-n/2,L^-}^{!}$.
% with integral principal part.
\end{theorem}

% see Corollary \ref{thm:main}
\begin{corollary}
\label{cor:main2intro}
Assume that $L\cong K\oplus U$ for some isotropic even
lattice $K$ of signature $(n-1,1)$ with $n\geq 3$.
Then there exists a sublattice  $K_0\subset K$ of the same level as $K$
such that
every meromorphic modular form $F$ with
respect to $\Gamma(L)$ whose divisor is a linear combination of special
divisors $Z(m,\mu)$
is  (up to a non-zero constant factor)
the Borcherds lift $\Psi(z,f)$ of some $f\in M_{1-n/2,K_0^-}^{!}$.
%h with integral principal part.
\end{corollary}

For even lattices of prime level, we are able to prove a converse theorem
without the hypothesis that $L$ splits a hyperbolic plane over $\Z$.

\begin{theorem}
\label{thm:main2intro}
Let $L$ be an even lattice of prime level $p$ and signature $(n,2)$. Assume that $n \geq 3$ and that the Witt rank of $L$ is $2$.
Then there exists a sublattice
$L_0\subset L$ of level $p$
such that
every meromorphic modular form $F$ with
respect to $\Gamma(L)$ whose divisor is a linear combination of special
divisors $Z(m,\mu)$
is  (up to a non-zero constant factor)
the Borcherds lift $\Psi(z,f)$ of some $f\in M_{1-n/2,L_0^-}^{!}$.
% with integral principal part.
\end{theorem}

To prove these results, we use a refinement of the approach of \cite{Br1}.
By extending the regularized theta lift of Borcherds to harmonic Maass forms,
one can construct a linear map
 \[
\Lambda: S_{1+n/2,L}\longrightarrow  \calH^{1,1}(X_\Gamma)
\]
from the space of cusp forms of weight $1+n/2$ with representation $\rho_L$
for the group $\Mp_2(\Z)$ to square integrable harmonic $(1,1)$ forms on
$X_\Gamma$, see Section \ref{sect:5}. In \cite{BF} it was shown that this map is adjoint
to the corresponding Kudla-Millson lift \cite{KM3}.
The converse theorem holds for meromorphic modular forms for $\Gamma$ if and only if
$\Lambda$ is injective on a certain subspace $S_{1+n/2,L}^+$ of $S_{1+n/2,L}$, see Theorem~\ref{lift+} and \cite[Theorem~5.11]{Br1}\footnote{Note that the definition of $S_{1+n/2,L}^+$ of the present paper differs from the definition given in \cite{Br1}. The space used in \cite{Br1} is ``too big'' in general
so that one only obtains a sufficient criterion.
But that difference can only occur for lattices that do not split a hyperbolic plane over $\Z$, which is why it did not play any role in that paper.}.
%If $L$ splits a hyperbolic plane over $\Z$, then $S_{1+n/2,L}^+=S_{1+n/2,L}$, but in general it is a true subspace

%To prove injectivity results for $\Lambda$, we employ its description in terms of Fourier expansions.
% given in \cite[Theorem 5.9]{Br1}.

Assume that $L$ is as in Theorem~\ref{thm:mainintro} and that $g\in \ker(\Lambda)$. Then,
since $L$ splits a hyperbolic plane over $\Z$, we have $S_{1+n/2,L}^+=S_{1+n/2,L}$.
By means of the description of $\Lambda$ in terms of Fourier expansions it can be deduced that certain Fourier coefficients of $g$ vanish. To show that {\em all\/} coefficients of $g$ vanish, we use a newform theory for vector valued modular forms for the Weil representation which we develop in Section~\ref{sect:3}.

The basic idea is as follows: If $H$ is a totally isotropic subgroup of the discriminant form $A=L'/L$, then $B:=H^\perp/H$ together with the induced quadratic form is also a discriminant form of size $|B|=|A|/|H|^2$.
There are intertwining operators for the Weil representations $\rho_A$ and $\rho_B$ which give rise to natural maps between $M_{k,B}$ and $M_{k,A}$ that are adjoint with respect to the Petersson inner product.
If $g\in M_{k,A}$ is in the image of the map from $M_{k,B}$, then $g$ is supported on $H^\perp$, that is, the components $g_\mu$ with $\mu\notin H^\perp$ vanish. Conversely, we show that any form in $M_{k,A}$ which is supported on $H^\perp$ must be in the image of the map.
More generally, we show that any element of $M_{k,A}$ which is supported on the union of orthogonal complements of isotropic subgroups $H_i\subset A$ must be a sum of forms induced from the corresponding smaller discriminant groups, see Theorems~\ref{thm:6} and \ref{thm:8}.

In Section \ref{sect:5.1} we employ this newform theory together with the equivariance of the map $\Lambda$ with respect to the action of the finite group $\Orth(L)^+/\Gamma$ to prove that $\Lambda$ is injective. Thereby we obtain
Theorem \ref{thm:mainintro}.

If $L$ does not split a hyperbolic plane over $\Z$, the map $\Lambda$ is not injective in general (see Section \ref{sect:5.2} for examples) and $S_{1+n/2,L}^+$ is a true subspace of $S_{1+n/2,L}$. Therefore, to prove Theorem \ref{thm:main2intro} we have to argue differently. First we employ the Fourier expansion of $\Lambda$
and the equivariance for the group $\Orth(L)^+/\Gamma$, to show that any $g\in S_{1+n/2,L}^+$ with $\Lambda(g)=0$ must be invariant under the action of $\Aut(L'/L)$. Then, using the defining relations of  $S_{1+n/2,L}^+$ and the Weil bound for the growth of Fourier coefficients of cusp forms, we infer that $g$ must actually vanish (Theorem \ref{thm:mainp}).

As an application of our injectivity results, we obtain information about the Picard groups of modular varieties. For instance, we shall prove the following result.

\begin{theorem}
\label{thm:rankestintro}
Assume that $L\cong D\oplus U(N)\oplus U$ for some positive definite even lattice $D$ of dimension $n-2>0$ and some positive integer $N$. Then the subgroup $\Psp(X_\Gamma)$ of the the Picard group of $X_\Gamma$ generated by the special divisors $Z(m,\mu)$ satisfies
\[
\rank(\Psp(X_\Gamma))= 1+ \dim(S_{1+n/2,L}).
\]
\end{theorem}

Note that the dimension of $S_{1+n/2,L}$ can be explicitly computed by means of the Selberg trace formula or the Riemann-Roch theorem,
% see Section \ref{sect:dimfor}.
see \cite[p.~228]{Bo3}.  The methods of the present paper can also be
used to prove injectivity results for other theta lifts of vector
valued modular forms such as \cite[Theorem 14.3]{Bo2}, see Section
\ref{sect:7.2}.

The present paper is organized as follows. In Section \ref{sect:2} we
collect some preliminaries on vector valued modular forms for the Weil
representation, and in Section \ref{sect:3} we develop a newform
theory in this setting.  Section \ref{sect:4} contains some facts on
modular varieties for orthogonal groups and special
divisors. Moreover,
%In Section \ref{sect:5}
we explain the lifting $\Lambda$ and the criterion for the converse
theorem. In Section \ref{sect:5.1} we consider lattices that split a
hyperbolic plane over $\Z$, and we prove Theorem \ref{thm:mainintro}
and Corollary \ref{cor:main2intro}. In Section \ref{sect:6} we
consider lattices of prime level and prove Theorem
\ref{thm:main2intro}. The applications to Picard groups and to other
theta liftings are considered in Section \ref{sect:7}.

I thank E.~Freitag and  N.~Scheithauer for many useful conversations on the content of this paper.

\section{Preliminaries}
\label{sect:2}
%[Metaplectic group.]

Here we briefly summarize some facts on lattices, discriminant
forms, and the Weil representation. For more details we refer to
\cite{Bo2}, \cite{Sch:Inv}, \cite{Sch}, \cite{Br1}.

Let $(L,Q)$ be a non-degenerate even lattice of signature $(b^+,b^-)$.
We denote by $(\cdot,\cdot)$ the bilinear form associated to the
quadratic form $Q$ (normalized such that $Q(x)=\frac{1}{2}(x,x)$).  We
write $L'$ for the dual lattice of $L$.
%The level of $L$ is defined as the smallest
%positive integer $N$ such that $NQ(x)\in \Z$ for all $x\in L'$.
The finite abelian group $L'/L$ is called the discriminant group of
$L$. Its order is equal to the absolute value of the Gram
determinant of $L$. We put $\sig(L)=b^+-b^-$.

Recall that a discriminant form is a finite abelian group $A$
together with a $\Q/\Z$-valued non-degenerate quadratic form
$x\mapsto Q(x)$, for $x\in A$ (see \cite{Ni}).
The level of $A$ is the smallest positive integer $N$ such that $NQ(x)\subset \Z$ for all $x\in A$.
If $L$ is
a non-degenerate even lattice then $L'/L$ is a discriminant form,
where the quadratic form is given by the mod $1$ reduction of the
quadratic form on $L'$. Conversely, every discriminant form can be
obtained in this way.
%\begin{comment}
The quadratic form on $L'/L$ determines the
signature of $L$ modulo $8$ by Milgram's formula:
%(see \cite{MH} Appendix 4):
\begin{equation}
\label{eq:milgram} \sum_{\lambda\in L'/L}
e(Q(\lambda))=\sqrt{|L'/L|} e( \sig(L)/8),
\end{equation}
where $e(z):=e^{2\pi i z}$ for $z\in \C$.
We define the signature $\sig(A)\in \Z/8\Z$ of a discriminant form
$A$ to be the signature of any even lattice with that discriminant
form.
%Moreover, we define the level of $A$ analogously.

\begin{comment}
If $A$ is a discriminant form, then we write $A^n$ for the subgroup
of elements that are $n$-th powers of elements of $A$. Moreover, we
write $A_n$ for the subgroup of elements of $A$ whose order divides
$n$. We have an exact sequence
\begin{align}
\label{eq;sequence} 0\longrightarrow A_n \longrightarrow A
\longrightarrow A^n \longrightarrow 0,
\end{align}
and $A^n$ is the orthogonal complement of $A_n$.
\end{comment}

Let $\H=\{\tau\in \C:\;\Im(\tau)>0\}$ be the complex upper half
plane.  We write $\Mp_2(\R)$ for the two-fold metaplectic
 cover of $\Sl(\R)$, realized as the group of pairs
  $(M,\phi(\tau))$ where $M=\kabcd\in\Sl_2(\R)$ and
$\phi:\H\to \C$ is a holomorphic function with
$\phi(\tau)^2=c\tau+d$.  The multiplication is defined by
\[
(M,\phi(\tau)) (M',\phi'(\tau))=(M M',\phi(M'\tau)\phi'(\tau)).
\]
%For $M\in\Gl_2^+(\R)$ we put
%\[
%M'=\left(M,\sqrt{c\tau+d}\right)\in\widetilde{\Gl}_2^+(\R)
%\]
%with the principal branch of the holomorphic square root.
We write $\Mp_2(\Z)$ for the integral metaplectic group, i.e.,
the inverse image of $\Gamma(1)=\Sl_2(\Z)$ under the covering map.
It is well known that
%the integral metaplectic group
$\Mp_2(\Z)$
is generated by $T=
\left( \kzxz{1}{1}{0}{1}, 1\right)$, and $S= \left(
\kzxz{0}{-1}{1}{0}, \sqrt{\tau}\right)$. One has the relations
$S^2=(ST)^3=Z$, where $Z=\left( \kzxz{-1}{0}{0}{-1}, i\right)$.
% is
%the standard generator of the center of $\Mp_2(\Z)$.

\subsection{The Weil representation}
\label{sect:2.1}
%[Define it as in \cite{Bo2}, \cite{Br1} for a lattice $(L,Q)$ of signature $(b^+,b^-)$.]

Let $A$ be a discriminant form.
Recall that there is a Weil representation
%associated with $A$
of $\Mp_2(\Z)$ on the group algebra $\C[A]$ (see e.g.~\cite{Bo2}, \cite{Br1}, \cite{We}).
We denote the
standard basis elements of $\C[A]$ by $\frake_\lambda$, $\lambda\in
A$, and write $\langle\cdot,\cdot \rangle$ for the standard scalar
product (antilinear in the second entry) such that $\langle
\frake_\lambda,\frake_\mu\rangle =\delta_{\lambda,\mu}$. The Weil
representation $\rho_A$ associated with $A$ is
the unitary representation of $\Mp_2(\Z)$ on the group algebra
$\C[A]$ defined by
\begin{align}
\label{eq:weilt}
\rho_A(T)(\frake_\lambda)&=e(Q(\lambda))\frake_\lambda,\\
\label{eq:weils}
\rho_A(S)(\frake_\lambda)&=\frac{e(-\sig(A)/8)}{\sqrt{|A|}} \sum_{\mu\in A} e(-(\lambda,\mu)) \frake_\mu.
%\end{align}
%\intertext{Note that}
%\begin{align}
%\label{eq:weilz} \rho_A(Z)(\frake_\lambda)&=e(-\sig(A)/4)
%\frake_{-\lambda}.
\end{align}
We have that $\rho_A(Z)(\frake_\lambda)=e(-\sig(A)/4)\frake_{-\lambda}$.
The automorphism group $\Aut(A)$ also acts on $\C[A]$ by
\begin{align}
\label{eq:weilh} \rho_A(h)(\frake_\lambda)&=\frake_{h\lambda}
\end{align}
for $h\in\Aut(A)$, and the actions of $\Mp_2(\Z)$ and $\Aut(A)$
commute.
It is well known that
$\rho_A$ is trivial on a subgroup of $\Mp_2(\Z)$ which is isomorphic via the projection map
to the principal congruence subgroup of level $N$, where $N$ is the level of $A$.
If $A=L'/L$ is the discriminant form associated to an even lattice $L$, then we briefly write $\rho_L$ for $\rho_{L'/L}$.
 Note that $\rho_L$ can be identified with a sub-representation of the usual Weil representation of $\Mp_2(\Z)$
on the space of Schwartz-Bruhat functions on $L\otimes \hat \Q$, see \cite{Ku:Integrals}.

%This representation is essentially the Weil representation attached to
%the quadratic module $(\calL,q)$ (see \cite{No}).

%If
%the signature of $A$ is even, then
%\eqref{eq:weilz} implies that
%$Z^2$ acts trivially. Hence, the Weil representation factors through
%$\Sl_2(\Z)$.

%
%If the signature of $A$ is odd, we notice that the level of $A$ must
%be divisible by $4$. This follows from the oddity formula (\cite{CS}
%p.~383 (30)) which implies that $A$ contains odd $2$-adic Jordan
%components.

%[Define it as in \cite{Bo2}, \cite{Br1} for a lattice $(L,Q)$ of signature $(b^+,b^-)$.]

\subsection{Vector valued modular forms}

%[Define harmonic Maass forms, (weakly) holmorphic modular forms, cusp forms. State dimension formula.]

%In many recent works vector valued modular forms associated to the
%Weil representation are considered (see e.g. \cite{Bo1}, \cite{Bo2},
%\cite{Br}, \cite{McG}, \cite{Sch}).
Let $k\in \frac{1}{2}\Z$.
%, and let
%$\Gamma'\subset \Mp_2(\Z)$ be a subgroup of finite index.
A holomorphic function $f:\H\to \C[A]$ is called a weakly holomorphic
modular form of
weight $k$ and type $\rho_A$ for the group $\Mp_2(\Z)$, if
\begin{align}
\label{eq:trans}
f(M\tau)= \phi(\tau)^{2k} \rho_A(M,\phi) f(\tau)
\end{align}
for all $(M,\phi)\in \Mp_2(\Z)$, and $f$ is meromorphic at the cusp $\infty$.
Such a function is called a holomorphic modular form if it
is actually holomorphic at $\infty$, and it is called a cusp form if it
vanishes at $\infty$.  We denote the complex vector space of such
weakly holomorphic modular forms by $M^!_{k,A}$.  We denote
the subspaces of holomorphic modular forms and cusp forms by
$M_{k,A}$ and $S_{k,A}$, respectively.
%If
%$\Gamma'=\Mp_2(\Z)$, we drop it from the notation and simply write for
%example $M_{k,A}$ for the space $M_{k,A}(\Gamma')$.
%It is an easy
%consequence of the the action of $Z$ in the Weil representation that
%$M^!_{k,A}=\{0\}$ unless $2k\equiv \sig(A)\pmod{2}$. Any $f\in M^!_{k,A}$ has a Fourier expansion of the form
%\begin{align}
%\label{eq:few}
%f(\tau)= \sum_{\mu\in A}\sum_{m\in Q(\mu)+\Z}c(m,\mu)q^m\frake_\mu,
%\end{align}
%where $q=e^{2\pi i\tau}$.
If $A=L'/L$ for an even lattice $L$, then we simply write  $M^!_{k,L}$ for $M^!_{k,L'/L}$, and analogously for other spaces of automorphic forms.

\begin{comment}
Recall that for $f,g\in M_{k,A}(\Gamma')$  the
Petersson scalar product is defined by
\begin{align}
(f,g)= \frac{1}{[\Mp_2(\Z):\Gamma']}\int_{\Gamma'\bs \H}  \langle f(\tau), g(\tau)\rangle \,v^k\,\frac{du\,dv}{v^2}.
\end{align}
Here $u$ denotes the real part and $v$ the imaginary part of $\tau\in \H$.
The Petersson scalar product converges when $\langle f(\tau), g(\tau)\rangle $ is a cusp form.
\texttt{Do we need that???}
\end{comment}

%\subsubsection{Harmonic Maass forms}

Following \cite{BF}, a smooth function $f:\H\to \C[A]$ is called a {\em harmonic Maass form} of weight $k$ with representation $\rho_A$ for $\Mp_2(\Z)$, if
\begin{itemize}
\item[(i)]
it satisfies the transformation law \eqref{eq:trans}
for all $(M,\phi)\in \Mp_2(\Z)$;
\item[(ii)] it satisfies $\Delta_k f =0$, where $\Delta_k$ is the
  hyperbolic Laplace operator in weight $k$;
\item[(iii)] it has at most linear exponential growth at the cusp.
\end{itemize}
The differential operator $\xi_k(f)= 2i v^k
\overline{\frac{\partial}{\partial \bar \tau}f} $ takes a harmonic
Maass form $f$ to a weakly holomorphic modular form of weight $2-k$
transforming with the dual of $\rho_A$.
Here $v$ denotes the imaginary part of $\tau\in \H$.

We let $H_{k,A}$ be the
subspace of those harmonic Maass forms of weight $k$ with
representation $\rho_A$ for $\Mp_2(\Z)$ for which $\xi_k(f)$ is a cusp
form. (This space was called $H_{k,A}^+$ in \cite{BF}.)  We
have the exact sequence
\begin{align}
\xymatrix{
0\ar[r]& M_{k,A}^! \ar[r]& H_{k,A} \ar[r]^{\xi_k}&  S_{2-k,A^-} \ar[r] & 0,
}
\end{align}
where $A^-$ denotes the discriminant form given by $A$ together with the quadratic form $-Q$.
%Note that $\rho_{A^-}$ can be identified with the dual representation of $\rho_A$.
Any $f\in H_{k,A}$ has a Fourier expansion of the form
\begin{align}
\label{weakmaass}
f(\tau)&= \sum_{\mu\in A}\sum_{m\in Q(\mu)+\Z} c^+(m,\mu) q^m\frake_\mu
+\sum_{\mu\in A}\sum_{\substack{m\in Q(\mu)+\Z \\ m<0}} c^-(m,\mu) \Gamma(1-k,4\pi |m|v)
q^m\frake_\mu,
\end{align}
where $\Gamma(a,t)$ denotes the incomplete gamma function. The finite sum
\begin{align*}
%\label{pp}
P_f(\tau)=\sum_{\mu\in A}\sum_{\substack{m\in Q(\mu)+\Z\\m<0}} c^+(m,\mu) q^m\frake_\mu
\end{align*}
is called the  {\em principal part} of $f$. It determines the growth of $f$ at the cusp $\infty$.
We say that $f$ has integral principal part,
if $c^+(m,\mu)$ for all $\mu\in A$ and all $m<0$.
% and $v$ is the imaginary part of $\tau$.

\begin{comment}
\subsubsection{The dimension formula}

\label{sect:dimfor}

Since $\rho_A$ factors through a finite quotient of $\Mp_2(\Z)$, it is
clear that the dimension of $M_{k,A}$ is finite. It can be computed
using the Riemann-Roch theorem or the Selberg trace formula, see \cite{Fi}.
Here, for simplicity, we assume that that $2k\equiv \sig(A)\pmod{4}$, since we will
only be interested in this case later.
%Then the following formula holds \cite[p.~228]{Bo2}:

Let $d=|A/\{\pm 1\}|$.
The $d$-dimensional
subspace $W=\Span\{\frake_\gamma+\frake_{-\gamma}:\;\gamma\in A \}$ of
$\C[A]$ is invariant under $\rho_A$, and
$\rho_A(Z)$ acts by multiplication with $e(-k/2)$ on $W$.
We denote by
$\rho$ the restriction of $\rho_A$ to $W$.
For a unitary matrix  $M$
of size $d$ with eigenvalues $e(\nu_j)$ and $0\leq \nu_j<1$,
%(for $j=1,\dots,d$),
we define
\[
\alpha(M)=\sum_{j=1}^d \nu_j.
\]
If $k>2$, then the dimension of $M_{k,A}$ is given by
\begin{equation}\label{dim1}
\dim_\C( M_{k,L} ) = d+dk/12-\alpha\left(e^{\pi i k/2}\rho(S)\right) - \alpha\left(\left( e^{\pi i k/3}\rho(ST)\right)^{-1}\right) -\alpha(\rho(T)),
\end{equation}
see \cite[p.~228]{Bo2}.
Furthermore, using Eisenstein series, it can be easily shown that the
codimension of $S_{k,A}$ in $M_{k,A}$ is equal to the cardinality of the set
%\begin{equation}\label{codim}
$\left\{\gamma\in A /\{\pm 1\}: \; Q(\gamma)\subset  \Z \right\}$,
%\end{equation}
see \cite{Br1} Chapter 1.2.3.

\end{comment}

\section{Newform theory for vector valued modular forms}
\label{sect:3}

Let $(A,Q)$ be a discriminant form. Let $H\subset A$ be an isotropic
subgroup, and write $H^\perp$ for its orthogonal complement in $A$.  Then
$B:=A_H=H^\perp/H$ together with the induced quadratic form is also a
discriminant form, and we have $|A|=|B|\cdot |H|^2$ and
$\sig(B)=\sig(A)$.

Let $f\in M_{k,A}$ and denote by $f_\mu$ for $\mu\in A$ the components
of $f$ with respect to the standard basis of $\C[A]$. Let $S\subset A$
be a subset. We say that $f$ is supported on $S$ if $f_\mu=0$ for all
$\mu \notin S$.

There are maps between the spaces $M_{k,A}$ and $M_{k,B}$, which we now describe.
The following result is Theorem 4.1 in \cite{Sch} (see also \cite[Lemma 5.6]{Br1} for a special case).

\begin{proposition}
\label{prop:1}
Let $g=\sum_{\nu\in B} g_\nu \frake_\nu\in M_{k,B}$. Then the $\C[A]$-valued function
\[
g\uparrow_H^A = \sum_{\mu\in H^\perp}
g_{\mu+H}\frake_\mu
\]
belongs to $M_{k,A}$. It is supported on $H^\perp$.
\end{proposition}

The next proposition generalizes \cite[Lemma 5.7]{Br1}.

\begin{proposition}
\label{prop:2}
Let $f=\sum_{\mu\in A} f_\mu \frake_\mu\in M_{k,A}$. Then the $\C[B]$-valued function
\[
f\downarrow_H^A = \sum_{\mu\in H^\perp}  f_\mu \frake_{\mu+H}
\]
belongs to $M_{k,B}$.
\end{proposition}

%\texttt{Give proof?}

The following proposition provides a converse for Proposition \ref{prop:1}.

\begin{proposition}
\label{prop:3}
Let $f\in M_{k,A}$, and assume that $f$ is supported on $H^\perp$.
Then $f_{\mu+\mu'}=f_\mu$ for all $\mu \in A$, $\mu'\in H$, and
\begin{align}
\label{eq:downup}
f= \frac{1}{|H|} f\downarrow_H^A\uparrow_H^A.
\end{align}
\end{proposition}

\begin{proof}
Using \eqref{eq:weils}, we see that
\[
f_\mu(-1/\tau) = \tau^k\frac{e(-\sig(A)/8)}{\sqrt{|A|}} \sum_{\nu\in A} e(-(\mu,\nu)) f_\nu(\tau)
\]
for all $\mu\in A$. Since $f$ is supported on $H^\perp$, we have
\[
f_\mu(-1/\tau) = \tau^k\frac{e(-\sig(A)/8)}{\sqrt{|A|}} \sum_{\nu\in H^\perp} e(-(\mu,\nu)) f_\nu(\tau).
\]
Consequently, we obtain for $\mu'\in H$ that
\begin{align*}
f_{\mu+\mu'}(-1/\tau) &= \tau^k\frac{e(-\sig(A)/8)}{\sqrt{|A|}} \sum_{\nu\in H^\perp} e(-(\mu+\mu',\nu)) f_\nu(\tau)
=f_{\mu}(-1/\tau).
\end{align*}
This proves that $f_{\mu+\mu'}=f_\mu$.
% for all $\mu \in A$, $\mu'\in H$.
The identity \eqref{eq:downup} is an immediate consequence.
\end{proof}

\begin{lemma}
\label{lem:4}
Let $f\in M_{k,A}$. If $G\subset A$ is any subgroup, then for all $\mu \in A$ we have
\begin{align*}
\frac{1}{|G|}
\sum_{\mu'\in G}f_{\mu+\mu'}(-1/\tau) = \tau^k\frac{e(-\sig(A)/8)}{\sqrt{|A|}} \sum_{\nu\in G^\perp} e(-(\mu,\nu)) f_\nu(\tau).
\end{align*}
\end{lemma}

\begin{proof}
By means of \eqref{eq:weils}, we see that
\begin{align*}
\frac{1}{|G|}
\sum_{\mu'\in G}f_{\mu+\mu'}(-1/\tau) &=  \tau^k\frac{e(-\sig(A)/8)}{|G|\sqrt{|A|}}\sum_{\mu'\in G} \sum_{\nu\in A}  e(-(\mu+\mu',\nu))f_\nu(\tau)\\
&=  \tau^k\frac{e(-\sig(A)/8)}{|G|\sqrt{|A|}} \sum_{\nu\in A}  e(-(\mu,\nu))f_\nu(\tau)\sum_{\mu'\in G}e(-(\mu',\nu)).\\
\end{align*}
Using orthogonality of characters, we find that the latter sum over $\mu'\in G$ vanishes unless $\nu\in G^\perp$, in which case it is equal to $|G|$. This proves the lemma.
\end{proof}

\begin{lemma}
\label{lem:5}
Let $C_1,\dots,C_m\subset A$ be subsets. For every subset $S\subset \{1,\dots,m\}$ put
$C(S)= \bigcap_{i\in S} C_i$. Then we have
\[
\chi_{C_1\cup\dots\cup C_m} = \sum_{\emptyset\neq S\subset\{1,\dots,m\}} (-1)^{|S|+1} \chi_{C(S)}.
\]
Here $\chi_C:A\to \{0,1\}$ is the characteristic function of $C\subset A$.
\end{lemma}

\begin{proof}
This set theoretic fact is well known. It can be proved by induction on $m$.
\end{proof}

\begin{theorem}
\label{thm:6}
Let $H_1,\dots,H_m\subset A$ be isotropic subgroups of prime order $p_i=|H_i|$ with $p_i\neq p_j$ for $i\neq j$.
For a subset $S\subset \{1,\dots,m\}$ let  $H_S:=\sum_{i\in S} H_i$.
If $f\in M_{k,A}$ is supported on $H_1^\perp \cup \dots\cup H_m^\perp$, then
\[
f= \sum_{\emptyset\neq S\subset\{1,\dots,m\}} (-1)^{|S|+1} \frac{1}{|H_S|} f \downarrow_{H_S}^A\uparrow_{H_S}^A.
\]
%where $H_S:=\sum_{i\in S} H_i$ for a subset $S\subset \{1,\dots,m\}$.
\end{theorem}

\begin{proof}
Note that the subgroup $H_S\subset A$ is isotropic for all $S\subset \{1,\dots,m\}$.
We prove the statement by induction on $m$. For $m=1$ it is Proposition \ref{prop:3}.

Now assume that $m>1$.
It follows from the transformation behavior \eqref{eq:weils} and the hypothesis on the support of $f$ that
\begin{align*}
f_\mu(-1/\tau) = \tau^k\frac{e(-\sig(A)/8)}{\sqrt{|A|}} \sum_{\nu\in H_1^\perp\cup\dots\cup H_m^\perp} e(-(\mu,\nu)) f_\nu(\tau)
\end{align*}
for $\mu \in A$. We employ Lemma \ref{lem:5} with $C_i=H_i^\perp$. Then we have
$C(S)=\bigcap_{i\in S} H_i^\perp= H_S^\perp$ and therefore
\begin{align*}
f_\mu(-1/\tau) &= \tau^k\frac{e(-\sig(A)/8)}{\sqrt{|A|}} \sum_{\emptyset\neq S\subset\{1,\dots,m\}} (-1)^{|S|+1}
\sum_{\nu\in H_S^\perp} e(-(\mu,\nu)) f_\nu(\tau).
\end{align*}
By means of Lemma \ref{lem:4} we obtain
\begin{align*}
f_\mu(\tau) &=  \sum_{\emptyset\neq S\subset\{1,\dots,m\}} (-1)^{|S|+1}
\frac{1}{|H_S|}\sum_{\mu'\in H_S}  f_{\mu+\mu'}(\tau).
\end{align*}

We now use the transformation behavior \eqref{eq:weilt} under $T$. Since $f_\mu(\tau+a)=e(aQ(\mu)) f_\mu(\tau)$ for $a\in \Z$ and since $H_S$ is isotropic, we find
\begin{align*}
f_\mu(\tau) &=  \sum_{\emptyset\neq S\subset\{1,\dots,m\}} (-1)^{|S|+1}
\frac{1}{|H_S|}\sum_{\mu'\in H_S}  e(a(\mu,\mu'))f_{\mu+\mu'}(\tau).
\end{align*}
If we sum over $a$ modulo the level of $A$, then on the right hand side all terms with $(\mu,\mu')\notin \Z$ cancel. Consequently,
\begin{align*}
f_\mu(\tau) &=  \sum_{\emptyset\neq S\subset\{1,\dots,m\}} (-1)^{|S|+1}
\frac{1}{|H_S|}\sum_{\substack{\mu'\in H_S\\ \mu'\perp\mu}} f_{\mu+\mu'}(\tau).
\end{align*}

Now assume that $\mu \in H_1^\perp$ and $\mu\notin H_i^\perp$ for $i=2,\dots,m$.
Then for $i\in \{2,\dots,m\}$ there is a $\mu_i\in H_i$ with $(\mu_i,\mu)\notin\Z$.
Since $p_i\mu_i=0$, we may assume that $(\mu_i,\mu)\equiv \frac{1}{p_i}\pmod{\Z}$.
%Then $\mu_i$ is a generator of $H_i$.
If $\mu'\in H_S$, then by the Chinese remainder theorem we see that $\mu'\perp \mu$ if and only if
$\mu'\in H_1$. Hence we obtain
\begin{align*}
f_\mu(\tau) &=  \sum_{\emptyset\neq S\subset\{1,\dots,m\}} (-1)^{|S|+1}
\frac{1}{|H_S|}\sum_{\substack{\mu'\in H_S\cap H_1}} f_{\mu+\mu'}(\tau)\\
&=\sum_{\emptyset\neq S\subset\{2,\dots,m\}} (-1)^{|S|+1}
\frac{1}{|H_S|} f_{\mu}(\tau)
+\sum_{1\in S\subset\{1,\dots,m\}} (-1)^{|S|+1}
\frac{1}{|H_S|}\sum_{\substack{\mu'\in  H_1}} f_{\mu+\mu'}(\tau),
\end{align*}
and therefore
\begin{align*}
\sum_{S\subset\{2,\dots,m\}} (-1)^{|S|}
\frac{1}{|H_S|} f_{\mu}(\tau)
=\sum_{1\in S\subset\{1,\dots,m\}} (-1)^{|S|+1}
\frac{1}{|H_S|}\sum_{\substack{\mu'\in  H_1}} f_{\mu+\mu'}(\tau).
\end{align*}
Here in the sum on the left hand side the case $S=\emptyset$ is included (with $H_S=\{0\}$).

There is a bijection between the subsets of $\{2,\dots,m\}$ and the subsets of  $\{1,\dots,m\}$ containing $1$ given by $S\mapsto \{1\}\cup S$.
Under this map we have $p_1|H_S|= |H_{\{1\}\cup S}|$.
Consequently, we obtain
\begin{align}
\label{eq:61}
f_{\mu}(\tau)
=\frac{1}{p_1}\sum_{\substack{\mu'\in  H_1}} f_{\mu+\mu'}(\tau).
\end{align}
Since $\mu \in H_1^\perp$ and $\mu\notin H_i^\perp$ for $i=2,\dots,m$, we also have for every $\mu_1\in H_1$ that $\mu+\mu_1 \in H_1^\perp$ and $\mu+\mu_1\notin H_i^\perp$ for $i=2,\dots,m$. This follows from the fact that $\mu_1\perp H_i$ for $i=1,\dots,m$.
Hence, \eqref{eq:61} implies for such $\mu$ and $\mu_1\in H_1$ that
\begin{align}
\label{eq:62}
f_{\mu+\mu_1}(\tau)=f_\mu(\tau).
\end{align}

Therefore the function
\[
\tilde f = f-\frac{1}{p_1} f\downarrow_{H_1}^A\uparrow_{H_1}^A\in M_{k,A}
\]
is supported on $H_2^\perp \cup\dots \cup H_m^\perp$.
By induction, we have
\[
\tilde f= \sum_{\emptyset\neq S\subset\{2,\dots,m\}} (-1)^{|S|+1} \frac{1}{|H_S|} \tilde f \downarrow_{H_S}^A\uparrow_{H_S}^A.
\]
Substituting the definition of $\tilde f$, we obtain
\begin{align*}
f&=\frac{1}{p_1} f\downarrow_{H_1}^A\uparrow_{H_1}^A +
\sum_{\emptyset\neq S\subset\{2,\dots,m\}} (-1)^{|S|+1} \frac{1}{|H_S|} f \downarrow_{H_S}^A\uparrow_{H_S}^A\\
&\phantom{=}{}+
\sum_{\emptyset\neq S\subset\{2,\dots,m\}} (-1)^{|S|} \frac{1}{p_1|H_S|} f \downarrow_{H_1}^A\uparrow_{H_1}^A \downarrow_{H_S}^A\uparrow_{H_S}^A.
\end{align*}
In the latter summand we note that
\[
\frac{1}{p_1|H_S|} f \downarrow_{H_1}^A\uparrow_{H_1}^A \downarrow_{H_S}^A\uparrow_{H_S}^A=
\frac{1}{|H_{\{1\}\cup S}|} f \downarrow_{H_{\{1\}\cup S}}^A\uparrow_{H_{\{1\}\cup S}}^A.
\]
Inserting this, we see that
\begin{align*}
f&=
\sum_{\emptyset\neq S\subset\{1,\dots,m\}} (-1)^{|S|+1} \frac{1}{|H_S|} f \downarrow_{H_S}^A\uparrow_{H_S}^A,
\end{align*}
concluding the proof of the theorem.
\end{proof}

\subsection{Newform theory for cyclic isotropic subgroups}

\label{sect:3.1}

For a positive integer $d$ we denote by $\Omega(d)$ the number of
prime factors of $d$ counted with multiplicities.  Let $e\in A$ be an
isotropic element of order $N\in \Z_{>0}$.  Then
$(e,A)=\frac{1}{N}\Z\subset \Q/\Z$. For $\lambda\in A$ the residue class
$N(e,\lambda)\in \Z/N\Z$ is well defined. We define the {\em content}
of $\lambda$ with respect to $e$ as
\begin{align}
\cont_e(\lambda):=\gcd (N(e,\lambda),N).
\end{align}
%Let $e'\in A$ such that $(e',e)=\frac{1}{N}+\Z$.
%\texttt{Later: Take $\ell\in L$ primitive isotropic of level $N$, and $\ell'\in L'$ with $(\ell',\ell)=1$. Put $e=\ell/N+L$ and $e'=\ell'+L$.}

For any divisor $d\mid N$, we consider the isotropic subgroup
\begin{align}
\label{eq:Id}
I_d=\langle \frac{N}{d}e\rangle\subset A
\end{align}
of order $d$. Its orthogonal complement is given by
\[
I_d^\perp= \{\lambda\in A:\quad d\mid N(e,\lambda)\in \Z/N\Z\}=\{\lambda\in A:\;d\mid \cont_e(\lambda) \}.
\]
We put $A(d)=I_d^\perp/I_d$. Then $|A|=d^2|A(d)|$.
%\texttt{Define content $\cont_e(\lambda):=\gcd (N(e,\lambda),N)=1$???}

\begin{comment}
\begin{proposition}
\label{prop:7}
Let $N'\in \Z_{>0}$ be a divisor of $N$. Assume that $f \in M_{k,A}$ is supported on
\[
\bigcup_{\substack{p\mid N'\\ \text{$p$ prime}}} I_p^\perp,
\]
i.e., $f_\lambda= 0 $ for all $\lambda\in A$ with $(\lambda,\frac{N}{N'}e)\notin\Z$.
Then
\[
f=-\sum_{1<d\mid N'} \mu(d)\frac{1}{d} f\downarrow_{I_d}^A\uparrow_{I_d}^A.
\]
Here $\mu$ denotes the Moebius function.
\end{proposition}
\end{comment}

\begin{proposition}
\label{prop:7}
%Let $N'\in \Z_{>0}$ be a divisor of $N$.
Assume that $f =\sum_{\lambda\in A}f_\lambda \frake_\lambda\in M_{k,A}$ is supported on
\[
\bigcup_{\substack{p\mid N\\ \text{$p$ prime}}} I_p^\perp,
\]
that is, $f_\lambda= 0 $ for all $\lambda\in A$ with $\cont_e(\lambda)=1$.
%$\gcd (N(e,\lambda),N)=1$.
Then
\[
f=-\sum_{1<d\mid N} \mu(d)\frac{1}{d} f\downarrow_{I_d}^A\uparrow_{I_d}^A.
\]
Here $\mu$ denotes the Moebius function.
\end{proposition}

\begin{proof}
We reduce the statement to Theorem \ref{thm:6} as follows.
Let $m$ be the number of distinct prime divisors of $N$, and let $p_1,\dots,p_m$ be the distinct primes dividing $N$.
Then $H_i:=I_{p_i}\subset A$ is an isotropic subgroup of prime order $p_i$.
For $S\subset\{ 1,\dots ,m\}$ we have
\[
H_S= \sum_{i\in S} H_i = \sum_{i\in S}I_{p_i} = I_d,
\]
where $d=|H_S|=\prod_{i\in S}p_i$. As $S$ runs through the non-empty subsets of $\{1,\dots,m\}$, the quantity $d  =|H_S|$ runs through the square-free non-trivial divisors of $N$.
Moreover, we have $(-1)^{|S|}= \mu(d)$. This proves the proposition.
\end{proof}

\begin{definition}
We define the subspace of oldforms in $M_{k,A}$ with respect to the cyclic isotropic subgroup $I_N=\langle e\rangle$ of $A$ to be
\[
M_{k,A}^{old}= \sum_{p\mid N} M_{k,A(p)} \uparrow_{I_p}^A.
\]
We define the space of newforms with respect to the cyclic isotropic subgroup $I_N$ to be the orthogonal complement of $M_{k,A}^{old}$.
\end{definition}

\begin{corollary}
\label{cor:9}
We have
\[
M_{k,A}^{old}= \left\{\text{$f\in M_{k,A}$}:\; \text{$f_\lambda= 0 $ for all $\lambda\in A$ with $\cont_e(\lambda)=1$}
% \gcd (N(e,\lambda),N)=1$
\right\}.
\]
\end{corollary}

We now give a refinement of Proposition \ref{prop:7}.

\begin{theorem}
\label{thm:8}
Let $t\in \Z_{\geq 0}$.
Assume that $f =\sum_{\lambda\in A}f_\lambda \frake_\lambda\in M_{k,A}$ is supported on
\begin{align}
\label{eq:suppass}
\bigcup_{\substack{d\mid N\\ \Omega(d)=t}} I_d^\perp,
\end{align}
that is, $f_\lambda= 0 $ for all $\lambda\in A$ unless $\Omega(\cont_e(\lambda))\geq t$.
%$\Omega((N(e,\lambda),N))\geq t$.
%\texttt{Define the `height' $\height_e(\lambda)=\Omega(\gcd(N(e,\lambda),N))\in \{0,\dots,\Omega(N)\}$ here???}
Then there exist modular forms $f_d\in M_{k,A(d)}$ such that
\begin{align}
\label{eq:fid}
f=\sum_{\substack{d\mid N\\ \Omega(d)\geq t}} f_d\uparrow_{I_d}^A.
\end{align}
\end{theorem}

\begin{proof}
We prove the proposition by induction on $t$.
For $t=0$ there is nothing to show.
For $t=1$ the assertion follows from Proposition \ref{prop:7}.

Now assume that $t>1$. The assumption \eqref{eq:suppass} on the support of $f$ for $t$ implies that $f$ is a fortiori supported on
\begin{align*}
\bigcup_{\substack{d\mid N\\ \Omega(d)=t-1}} I_d^\perp.
\end{align*}
By induction, there exist modular forms $g_d\in M_{k,A(d)}$ such that
\begin{align}
\label{eq:fid2}
f=\sum_{\substack{d\mid N\\ \Omega(d)\geq t-1}} g_d\uparrow_{I_d}^A.
\end{align}

Let $d_0$ be a divisor of $N$ with $\Omega(d_0)=t-1$. We claim that
 $g_{d_0}$ is an oldform in $M_{k,A(d_0)}$ with respect to the cyclic subgroup
$\langle e+I_{d_0}\rangle \subset A(d_0)$ of order $N_0:=N/d_0$.
In fact, let $\nu\in A(d_0)$ with
$\cont_{e+I_{d_0}}(\nu)=1$.
%$(N_0(e+I_{d_0},\nu),N_0)=1$.
Then, if $\mu\in I_{d_0}^\perp$ with $\mu\mapsto \nu$ under the natural map
$I_{d_0}^\perp\to A(d_0)$, we have
$\cont_e(\mu)=d_0$.
%$d_0\mid N(e,\mu)$ and $(N_0(e,\mu),N_0)=1$.
Therefore $\mu$ does not belong to $I_d^\perp$ for any $d\mid N$ different from $d_0$ with
$\Omega(d)\geq t-1$.
Hence, identity \eqref{eq:fid2} implies that
\begin{align*}
(g_{d_0})_\nu &= (g_{d_0}\uparrow_{I_{d_0}}^A)_\mu\\
&=\sum_{\substack{d\mid N\\ \Omega(d)\geq t-1}} (g_d\uparrow_{I_d}^A)_\mu\\
&=f_\mu .
\end{align*}
Since $\Omega(d_0)=t-1$, we have by the hypothesis on $f$ that $f_\mu=0$, and therefore $   (g_{d_0})_\nu=0$.
Therefore, $g_{d_0}$ is an oldform in $M_{k,A(d_0)}$ with respect to the subgroup
$\langle e+I_{d_0}\rangle \subset A(d_0)$.

Consequently, for  any $d\mid N$ with $\Omega(d)=t-1$ we find by Corollary \ref{prop:7} that there exist modular forms $g_{d,p}\in M_{k,A(dp)}$ such that
\[
g_d= \sum_{\substack{p\mid N/d\\ \text{$p$ prime}}} g_{d,p}\uparrow_{J_p}^{A(d)}.
\]
Here $J_p$ denotes the isotropic subgroup
\[
J_p=\langle \frac{N}{dp}e+I_d\rangle\subset A(d)
\]
of order $p$. Note that $J_p^\perp/J_p\cong A(dp)$ and
\[
g_d\uparrow_{I_d}^A= \sum_{\substack{p\mid N/d\\ \text{$p$ prime}}} g_{d,p}\uparrow_{I_{dp}}^A.
\]
Substituting this into \eqref{eq:fid2}, we obtain the assertion for $t$.
\end{proof}

\begin{remark}
\label{rem:9}
Let $t\in \Z_{\geq 0}$.
Assume that $f =\sum_{\lambda\in A}f_\lambda \frake_\lambda\in M_{k,A}$ is given by
\begin{align*}
f=\sum_{\substack{d\mid N\\ \Omega(d)\geq t}} f_d\uparrow_{I_d}^A.
\end{align*}
for some  $f_d\in M_{k,A(d)}$ as in Theorem \ref{thm:8}.
Let $d_0\mid N $ with $\Omega(d_0)=t$,
and let $\mu\notin I_d^\perp$ for all $d\mid N$, $d\neq d_0$ with $\Omega(d)\geq t$.
Then $f_\mu = (f_{d_0}  \uparrow_{I_{d_0}}^A)_\mu$ and $f_{\mu+\mu'}=f_\mu$ for all $\mu'\in I_{d_0}$.
\end{remark}

\section{Modular varieties and special cycles}
\label{sect:4}

%Here we recall some facts on Borcherds products in a setup which is convenient for the present paper. See \cite{Bo1}, \cite{Bo2} for details.

%\texttt{Modular varieties and modular forms}

Let $(V,Q)$ be a rational quadratic space of signature $(n,2)$ and let $\Orth(V)$ be its orthogonal group viewed as an algebraic group over $\Q$. We realize the corresponding hermitian symmetric space as the Grassmannian
\begin{align}
\D= \{ z\subset V(\R):\; \text{$\dim(z)=2$ and $Q\mid_z<0$}\}
\end{align}
of negative definite oriented subspaces of $V(\R)$ of dimension $2$.
Note that $\D$ has two connected components given by the two possible choices of an orientation of $z\subset V(\R)$. We fix one component and denote it by $\D^+$.
The group $\Orth(V)(\R)$ acts transitively on $\D$. A subgroup $\Orth(V)(\R)^+$ of index $2$ (the subgroup of elements $\Orth(V)(\R)$ whose spinor norm has the same sign as the determinant)
acts transitively on $\D^+$.

The complex structure on $\D$ is most easily realized as follows. We extend the bilinear form on $V$ to a $\C$-bilinear form on $V(\C)=V\otimes_\Q\C$.
The open subset
\begin{align}
\calK=\{ [Z]\in P(V(\C)):\; \text{$(Z,Z)=0$ and $(Z,\bar Z)<0$}\}
\end{align}
of the zero quadric of the projective space $P(V(\C))$ of $V(\C)$ is isomorphic to $\D$ by mapping $[Z]$ to the subspace $\R\Re(Z)+\R\Im(Z)\subset V(\R)$ with the appropriate orientation.

We choose an isotropic vector $\ell\in V$ and a vector $\ell'\in V$ such that $(\ell,\ell')=1$. The rational quadratic space
$V_0:=V\cap \ell^\perp\cap \ell'{}^\perp$ has signature $(n-1,1)$. The tube domain
\begin{align}
\label{eq:td}
\calH=\calH_{\ell,\ell'}=\{ z\in V_0\otimes_\Q\C:\; \text{$Q(\Im(z))<0$}\}
\end{align}
is isomorphic to $\calK$ by mapping $z\in \calH$ to the class in $P(V(\C))$ of
\begin{align*}
w(z)=z+\ell'-\left(Q(z)-Q(\ell')\right)\ell.
\end{align*}
%The domain $\calH$ can be viewed as a generalized complex upper  half plane.
The linear action of $\Orth(V)(\R)$ on $V(\C)$ induces an action on $\calH$ by fractional linear transformations.
If $\gamma\in \Orth(V)(\R)$, we have
$\gamma w(z) = j(\gamma, z) w(\gamma z)$
for an automorphy factor $j(\gamma,z)=(\gamma w(z), \ell)$.
We write $\calK^+$ and $\calH^+$ for the connected components of $\calK$ and $\calH$, respectively,  corresponding to $\D^+$ under the above isomorphisms.

\begin{comment}
The function
\begin{align}
\label{eq:pet1}
Z\mapsto -\frac{1}{2}(Z,\bar Z)=-(Y,Y)=:|Y|^2
\end{align}
on $V(\C)$ defines a hermitian metric on the tautological line bundle $\calL$ over $\calK$, where $Y=\Im(Z)$.
Its first Chern form
\begin{align}
\label{eq:omega}
\Omega=-dd^c\log|Y|^2
\end{align}
is $\Orth(V)(\R)$-invariant and positive.  It corresponds to an invariant K\"ahler metric on $\D\cong \calK$.
%The corresponding invariant volume form  is given by $\frac{1}{m!}\Omega^m$.
%and gives rise to an invariant volume form $d\mu(z)=\Omega^n$.
Note that for $z\in \calH$ we have
\begin{align}
-\frac{1}{2}(w(z),\overline{w(z)})&=-(\Im(z),\Im(z)),\\
|\Im(\gamma z)|^2&= |j(\gamma,z)|^{-2}|\Im(z)|^2 .
\end{align}
\end{comment}

Let $L\subset V$ be an even lattice. Let $L'$ be its dual and write $A=L'/L$ for the discriminant group. We denote by $\Orth(L)$ the orthogonal group of $L$ and put $\Orth(L)^+=\Orth(L)\cap \Orth(V)(\R)^+$. % This is an arithmetic subgroup of $\Orth(V)(\R)$ which acts on $\D^+$ with finite covolume.
The kernel $\Gamma=\Gamma(L)$ of the natural map
\[
\Orth(L)^+\longrightarrow \Aut(A)
\]
is called the {\em discriminant kernel} subgroup of $\Orth(L)^+$.
%Throughout we let $\Gamma\subset \Orth(L)^+$ be a normal subgroup, which is contained in $\Gamma(L)$.
We consider the modular variety
$X_\Gamma = \Gamma\bs \D^+.$
By the theory of Baily-Borel, it carries the structure of a quasi-projective algebraic variety.

A meromorphic modular form of weight $k\in \Z$ for $\Gamma$ is a meromorphic function $\Psi$ on
$\calH^+$ which satisfies
%\begin{enumerate}
%\item[(i)]
$\Psi(\gamma z)=j(\gamma,z)^k\Psi(z)$ for all $\gamma\in \Gamma$ and which is
%\item[(ii)] $\Psi$ is
meromorphic  at the boundary.
%\end{enumerate}
%
%The last condition is trivially fulfilled if $V$ is anisotropic over $\Q$. By the Koecher principle, it is also automatically fulfilled if the Witt rank of $V$ over $\Q$ (i.e.~the dimension of a maximal totally isotropic subspace over $\Q$) is smaller than $n$.
The transformation law can be relaxed by allowing characters or multiplier systems, see e.g.~\cite{Br2}, Chapter 3.3.

%Modular forms of weight $w$ can be viewed as global sections of the
%coherent sheaf
%$\calM_w$
%of modular forms of weight $w$.
Recall that the Petersson norm of a modular form $\Psi$ of weight $k$ is given by
\begin{align*}
\|\Psi(z)\|_{Pet}= |\Psi(z)|\cdot |y|^{k},
\end{align*}
where $|y|^k=|(\Im(y),\Im(y))|^{k/2}$. Since
$|\Im(\gamma z)|^2= |j(\gamma,z)|^{-2}|\Im(z)|^2$, the Petersson norm
defines a $\Gamma$-invariant function on $\calH^+$.
The differential form
\[
\Omega = -dd^c\log |y|^2
\]
on $\calH^+$ is invariant under $\Orth(V)(\R)^+$ and positive. It corresponds to the  invariant K\"ahler metric on $\calH^+$, which unique up to a positive scalar factor. Moreover, it is the first Chern form of the sheaf of modular forms of weight $1$.

%The first Chern form of the line bundle $\calM_w$ with the Petersson metric is $w\Omega$.
%where $y=\Im(z)$.
%It defines a function on $X_K$.

There are special divisors on $X_\Gamma$, given by quadratic subspaces of $V$ of signature $(n-1,2)$, see e.g.~\cite{Bo2}, \cite{Ku:Duke}, \cite{GN}.
If $\lambda\in V$ is a vector of positive norm, then
\[
Z(\lambda)= \{ z\in \D:\; z\perp \lambda\}
\]
defines an analytic  divisor on $\D$. For $\mu\in L'/L$ and $m\in \Q_{>0}$ the special divisor of discriminant $(m,\mu)$ is
given by
\begin{align}
\label{eq:heeg}
Z(m,\mu)=\sum_{\substack{\lambda\in L+\mu\\ Q(\lambda)=m}} Z(\lambda).
\end{align}
It is a $\Gamma$-invariant divisor on $\D$. Since $\Gamma$ acts on the vectors of fixed norm in $L'$ with finitely many orbits, $Z(m,\mu)$ descends to an algebraic divisor on $X_\Gamma$.
Note that $Z(m,\mu)=0$ if $m\notin Q(\mu)+\Z$, and that $Z(m,\mu)=Z(m,-\mu)$.

%\texttt{Borcherds' Theorem}

In the present paper we are interested in those meromorphic modular forms for the group
$\Gamma$ which are obtained as Borcherds lifts of weakly holomorphic modular forms as described in the introduction,
see Theorem \ref{thm:bo}. Their zeros and poles lie on
special
divisors.
Note that the Borcherds lift is equivariant with respect to the actions of $\Orth(L)^+$ on $M_{1-n/2,L^-}^!$ and
on meromorphic modular forms for $\Gamma$.

%Therefore,  $\Psi(z,f)$ is actually modular for the stabilizer in $\Orth(L)^+$ of the weakly holomorphic modular form $f$.

\begin{comment}
The idea of the proof is to construct the logarithm of  $\|\Psi(z,f)\|_{Pet}$ by means of a regularized theta lift.
For $\tau=u+iv\in \H$ and $z\in \D$ we define the Siegel theta function of the lattice $L$ by
\[
\Theta_L(\tau,z)= v\sum_{\lambda\in L'} e\big( Q(\lambda_{z^\perp})\tau +Q(\lambda_z)\bar\tau\big)\frake_\lambda.
\]
In the variable $z$ it is a $\Gamma$-invariant function, and in $\tau$
it transforms as a non-holomorphic modular form of weight $n/2-1$ with
representation $\rho_A$ for $\Mp_2(\Z)$.  For $f\in M_{1-n/2,A^-}^!$
we consider the theta integral
\[
\Phi(z,f)=\int_{\Mp_2(\Z)\bs \H}^{reg}\langle f(\tau),\overline{\Theta_L(\tau,z)}\rangle \,\frac{du\,dv}{v^2},
\]
where the integral has to be regularized as in \cite{Bo2}, \cite{BF}.
It can be proved that $\Phi(z,f)=-4\log\|\Psi(z,f)\|_{Pet}+C$ for a constant $C$.
Then the properties of the theta lift can be used to derive the theorem.
\end{comment}

\subsection{Chern classes of special divisors and the converse theorem}
\label{sect:5}

Throughout we put $\kappa=1+n/2$.
Here we consider the question whether there is a converse to Borcherds' Theorem:
%\begin{question}
%\label{qestion:1}
Assume that $F$ is a meromorphic modular
form for the group $\Gamma$ whose zeros and poles are supported on special divisors, that is,
\begin{align}
\label{eq:divf}
\dv(F)=\frac{1}{2}\sum_\mu\sum_{m>0} c(-m,\mu) Z(m,\mu).
\end{align}
%(where $c(n,h)=c(n,-h)$ without loss of generality).
Is there a weakly holomorphic form $f\in M^!_{2-\kappa,L^-}$ whose Borcherds lift $\Psi(z,f)$ as in Theorem~\ref{thm:bo}
is equal to $F$?
%\end{question}
We recall and refine the approach to this question
developed in \cite{Br1}.

%\texttt{The lift $\Lambda$, briefly: connection to KM-lift as in \cite{BF} }

Let $\calH^{1,1}(X_\Gamma)$ be the space of square integrable harmonic
differential forms of Hodge type $(1,1)$ on $X_\Gamma$.  Recall from
\cite[Chapter 5.1]{Br1} that there is a linear map $S_{\kappa,L}\to
\calH^{1,1}(X_\Gamma)$, which can be obtained from the regularized
theta lift on harmonic Maass forms as follows.  For $\tau=u+iv\in \H$
and $z\in \D$ we define the Siegel theta function associated with the
lattice $L$ by
\begin{align}
\label{eq:c1}
\Theta_L(\tau,z)= v\sum_{\lambda\in L'} e\big( Q(\lambda_{z^\perp})\tau +Q(\lambda_z)\bar\tau\big)\frake_{\lambda+L}.
\end{align}
In the variable $z$ it is a $\Gamma$-invariant function, and in $\tau$
it transforms as a non-holomorphic modular form of weight $\kappa-2$ with
representation $\rho_L$ for $\Mp_2(\Z)$.

Let $f\in H_{2-\kappa,L^-}$ and denote the Fourier coefficients of $f$ by $c^\pm(m,\mu)$ as in \eqref{weakmaass}.
We consider the theta integral
\begin{align}
\label{eq:c2}
\Phi(z,f)=\int_{\Mp_2(\Z)\bs \H}^{reg}\langle f(\tau),\overline{\Theta_L(\tau,z)}\rangle \,\frac{du\,dv}{v^2},
\end{align}
where the integral has to be regularized as in \cite{Bo2}, \cite{BF}. It turns out that
$\frac{1}{2}\Phi(z,f)$ is a logarithmic Green function for the divisor
\begin{align}
\label{eq:defzf}
Z(f)= \frac{1}{2}\sum_{\mu\in L'/L}\sum_{m>0} c^+(-m,\mu) Z(m,\mu)
\end{align}
on $X_\Gamma$, see \cite[Theorem 2.12]{Br1}. Moreover, the differential form
$dd^c\Phi(z,f)$ can be continued to a smooth square integrable harmonic $(1,1)$-form.

If $f$ is actually weakly holomorphic, one can show that
$dd^c\Phi(z,f)=c^+(0,0)\Omega$, see e.g.~\cite[Theorem 6.1]{BF}.
This implies that there is a
meromorphic modular form $\Psi(z,f)$ for $\Gamma$ as in Theorem \ref{thm:bo} such that
 $-4\log\|\Psi(z,f)\|_{Pet}$ is up to a constant equal to $\Phi(z,f)$.
(This can be actually used to prove Theorem \ref{thm:bo} up to the infinite product expansion.)
Hence, if $g\in S_{\kappa,L}$, we can pick a harmonic Maass form $f\in H_{2-\kappa,L^-}$ with vanishing constant term $c^+(0,0)$ such that $\xi(f)=g$, and define
\[
\Lambda(g,z)=dd^c\Phi(z,f).
\]
We obtain a well defined
linear map $\Lambda: S_{\kappa,L}\to \calH^{1,1}(X_\Gamma)$. Alternatively, $\Lambda$ can be constructed by integrating $g$ against the Kudla-Millson theta function \cite{KM3} associated to $L$,
%and adding a suitable multiple of the K\"ahler form $\Omega$,
see \cite[Theorem 6.1]{BF}.

To describe the map $\Lambda$ in terms of Fourier expansions,
we view the elements of $\calH^{1,1}(X_\Gamma)$ as $\Gamma$-invariant differential forms in a tube domain model for $\D^+$. To this end, let $\ell\in L$ be a primitive isotropic vector and let $\ell'\in L'$ such that $(\ell,\ell')=1$. Let $V_0=V\cap \ell^\perp\cap \ell'{}^\perp$, and write $\calH$ for the corresponding tube domain realization of $\D$ as in \eqref{eq:td}.
Note that the lattice $L\cap V_0$ is isometric to $K=(L\cap\ell^\perp)/\Z\ell$.
(Warning: in general $K'$ is not contained in $L'$.)
For $\delta\in L'$, we denote by $\delta\mid L\cap \ell^\perp$ the restriction of $\delta\in\Hom(L,\Z)$ to $L\cap \ell^\perp$. We consider $\gamma\in K'$ as an element of $\Hom(L\cap\ell^\perp,\Z)$ via the quotient map $ L\cap\ell^\perp\to K$.

Following \cite[(3.25)]{Br1}, for $a,b\in \R$ we define the special function
\begin{align}
\label{DefV}
\calV_\kappa(a,b) = \int\limits_0^\infty \Gamma(\kappa-1, a^2 y) e^{-b^2 y -1/y}y^{-3/2}\,dy.
\end{align}
Then for $z=x+iy\in \calH^+$ and for $\lambda\in K'$, the differential form
\begin{align*}
dd^c \calV_{\kappa} \left( \pi  |\lambda| |y|,\,\pi  (\lambda,y)\right) e((\lambda,x))
\end{align*}
is harmonic and invariant under translations by $K$. It is exponentially decreasing in $a$ and $b$,
see \cite[Section 3.2]{Br1}.

\begin{theorem}
\label{thm:lift}
The map $\Lambda: S_{\kappa,L}\to \calH^{1,1}(X_\Gamma)$ has the following properties:
\begin{enumerate}
\item[(i)] If $g\in S_{\kappa,L}$ with Fourier expansion $g= \sum_\mu\sum_\mu b(m,\mu)q^m\frake_\mu$, then the Fourier expansion of $\Lambda(z,g)$ is given by
\begin{align*}
\Lambda(z,g) = \Lambda_0(y,g)-&2^{2-\kappa}\pi^{1/2-\kappa}
\sum_{\substack{\lambda\in K'\\ Q(\lambda)> 0}} |\lambda|^{-n}
 \sum_{d\mid\lambda} d^{n-1} \!\!\!\sum_{\substack{\delta\in L'/L \\ \delta\mid L\cap \ell^\perp=\lambda/d+K}}  \!\!\!e(d(\delta,\ell'))\\
 &{} \times b(Q(\lambda)/d^2, \delta) dd^c  \calV_{\kappa} \left( \pi  |\lambda| |y|, \pi  (\lambda,y)\right)e((\lambda,x)).
\end{align*}
Here  $|\lambda|=|(\lambda,\lambda)|^{1/2}$ and the sum $\sum_{d\mid\lambda}$ runs through all positive integers $d$ such that $\lambda/d\in K'$.
 Moreover, the $0$-th coefficient $\Lambda_0(y,g)$ is a certain $(1,1)$-form which is independent of $x$.
\item[(ii)]
If $f\in H_{2-\kappa,L^-}$ with Fourier coefficients $c^\pm(m,h)$
such that $\xi(f)=g$, then
$\Lambda(z,\xi(f))$
is a square integrable harmonic representative for the Chern class in $H^2(X_\Gamma,\C)$ of the divisor $2Z(f)$.
\item[(iii)]
For $\gamma\in \Orth(L)^+$ we have
\[
\Lambda(z,\gamma.g)= \Lambda(\gamma z,g).
\]
Here $g\mapsto \gamma.g$ denotes the action of $\Orth(L)^+$ on $S_{\kappa,L}$ via its action on $\C[L'/L]$ through $\Orth(L)^+\to \Aut(L'/L)$.
\end{enumerate}
\end{theorem}

\begin{proof}
The first assertion is the first part of \cite[Theorem 5.9]{Br1}. The second assertion follows from
\cite[Theorem 5.5]{Br1} or \cite[Theorem 7.3]{BF}.
The third statement follows from the construction of $\Lambda$ by means of the theta lift \eqref{eq:c2} and the corresponding equivariance property of the Siegel theta function \eqref{eq:c1}.
\end{proof}

Let $\Div(X_\Gamma)$ be the group of divisors of $X_\Gamma$ and put $\Div(X_\Gamma)_\C=\Div(X_\Gamma)\otimes\C$. We define a subspace of $H_{2-\kappa,L^-}$ by
\begin{align}
\label{bi38}
N_{2-\kappa,L^-} & =\{\text{$f\in H_{2-\kappa,L^-}$}:\;\text{$Z(f)=0\in \Div(X_\Gamma)_\C$}\}.
\end{align}
It follows from the construction of the map $\Lambda :S_{\kappa,L}\to \calH^{1,1}(X_\Gamma)$
and \cite[Theorem 4.23]{Br1}
% and Theorem~\ref{thm:lift}
% and the construction of $\Lambda$
that $\xi(   N_{2-\kappa,L^-})$ is contained in the kernel of $\Lambda$.
We let $S_{\kappa,L}^+$ be the orthogonal complement of $\xi(   N_{2-\kappa,L^-})$ with respect to the Petersson scalar product. Notice that the spaces $N_{2-\kappa,L^-}$ and $S_{\kappa,L}^+$ in general really depend on $L$ and not only on $L'/L$. They are stable under the action of $\Orth(L)^+$.
If the lattice $L$ splits a hyperbolic plane over $\Z$, then $N_{2-\kappa,L^-}=0$, but in general it can be non-zero.
The constant term $c^+(0,0)$ automatically vanishes for any $f\in N_{2-\kappa,L^-} $.

We write $\Lambda^+$ for the restriction of $\Lambda$ to $S_{\kappa,L}^+$.
The following theorem gives a {\em necessary and sufficient\/} criterion for the converse theorem. It is a refinement of
\cite[Theorem 5.11]{Br1}.
%\texttt{It might be better to explain this in more detail with a diagram???}

\begin{theorem}\label{lift+}
Suppose that $n\geq 2$ and that $n$ is greater than the Witt rank of $V$.
The following are equivalent:

i) The map $\Lambda^+: S_{\kappa,L}^+\to \calH^{1,1}(X_\Gamma)$
is injective.

ii) Every meromorphic modular form $F$ with respect to $\Gamma$
whose divisor is a linear combination of special divisors as in
\eqref{eq:divf}
%($c(\lambda)\in \Z$ with $c(\lambda)=c(-\lambda)$ and $c(\lambda)=0$ for all but finitely many $\lambda$),
is (up to a non-zero constant factor)
the Borcherds lift $\Psi(z,f)$ of a weakly
holomorphic modular form $f\in M_{2-\kappa,L^-}^{!}$ with integral principal part.
\end{theorem}

\begin{proof}
  First, assume that $\Lambda^+$ is injective, and let $F$ be as in
  (ii). Since there always exist Borcherds products for $\Gamma$ of
  non-zero weight, we may assume that $F$ has weight $0$.  According
  to \cite[Theorem 4.23]{Br1} there exists an $f\in H_{2-\kappa,L^-}$
  with vanishing constant term $c^+(0,0)$ and integral principal part
  such that $\Phi(z,f)$ is equal to $-4\log|F(z)|$ up to a constant.
  Then $\Lambda(\xi(f))=dd^c \Phi(z,f)=0$, and (i) implies that
  $\xi(f)\in \xi(N_{2-\kappa,L^-})$.  Hence, there exists a $h\in
  N_{2-\kappa,L^-}$ with integral principal part such that
  $\xi(f)=\xi(h)$.  Consequently, $f_0:=f-h$ belongs to
  $M^!_{2-\kappa,L^-}$, and satisfies $Z(f_0)=Z(f)=\dv(F)$.  The
  Borcherds lift of $f_0$ is equal to $F$ up to a constant factor.

  Now assume that (ii) holds, and let $g\in S_{\kappa,L}^+$ such that
  $\Lambda^+(g)=0$.  Let $f\in H_{2-\kappa,L^-}$ with vanishing constant
  term such that $\xi(f)=g$.
%\texttt{Without loss of generality we may assume that $f$ has integral principal part.}
The fact that
  $dd^c\Phi(z,f)=\Lambda^+(g)=0$ implies that there exists a meromorphic
  modular form $F$ of weight $0$ for $\Gamma$ (with a character of
  finite order) such that $-4\log|F|=\Phi(z,f)$, cf.~\cite[Lemma 6.6]{Br3}.
  In particular, the
  divisor of $F$ is supported on special divisors, and therefore (ii)
  implies that there exists a $f_0\in M^!_{2-\kappa,L^-}$ such that
  $\Phi(z,f_0)=\Phi(z,f)$.  Consequently, $f-f_0\in N_{2-\kappa,L^-}$ and
  $g=\xi(f-f_0)$. Since $g\in S_{\kappa,L}^+$, we obtain that $g$ must
  vanish.
%
%That the first statement implies the second can be proved as in \cite[Theorem 5.11]{Br1}.
%Here we prove that the second implies the first. \texttt{Put in !!!}
\end{proof}

\section{Lattices that split a hyperbolic plane over $\Z$}
\label{sect:5.1}

Here we use the newform theory of Section \ref{sect:3}
and the criterion given in Theorem \ref{lift+}
%a criterion of \cite{Br1}
to prove a converse theorem for lattices that split a hyperbolic plane over $\Z$. We continue to use the notation of the previous section.

If $M$ is a lattice equipped with a quadratic form $q$, and $N$ is a non-zero integer, we write
$M(N)$ for the lattice given by $M$ as a $\Z$-module, but equipped with the rescaled quadratic form $N\cdot q$. We have $M(N)'=\frac{1}{N}M'$.
We let $U$
be the lattice $\Z^2$ with the quadratic form $q((a,b))=ab$.
Up to isometry this is the unique unimodular even lattice of signature $(1,1)$.
WE call any lattice isomorphic to $U$ a hyperbolic plane.

\begin{lemma}
\label{lem:6.1}
Let $L$ be an even lattice of level $N$.
Let $\ell\in L$ be a primitive isotropic vector such that $(\ell,L)=N\Z$.
Then there exists an isotropic vector $\tilde \ell \in L$ with $(\tilde \ell,\ell)=N$ such that $L=K\oplus \Z\tilde \ell\oplus \Z\ell$, where $K=L\cap \ell^\perp\cap\tilde\ell^\perp$. In particular, $L\cong K\oplus U(N)$.
\end{lemma}

\begin{proof}
Let $\ell'\in L'$ such that $(\ell',\ell)=1$, and put $\tilde \ell= N(\ell'-Q(\ell')\ell)$. Then $\tilde \ell$ is isotropic and satisfies $(\tilde \ell,\ell)=N$.
Since $L$ has level $N$, we have $NL'\subset L$ and $NQ(\ell')\in \Z$. Consequently, $\tilde \ell$ belongs to $L$.
The splitting $L=K\oplus \Z\tilde \ell\oplus \Z\ell$ can be proved as in \cite[Proposition 2.2]{Br1}.
\end{proof}

We now assume that the lattice $L\subset V$ is of the form
$L\cong D\oplus U(N)\oplus U$ for some positive definite even lattice $D$ of dimension $n-2$.
We put
\[
A=L'/L\cong D'/D\oplus U(N)'/U(N).
\]
We have $U(N)'/U(N)\cong (\Z/N\Z)^2$ and the
automorphism group of $U(N)'/U(N)$ contains $(\Z/N\Z)^\times$.
In fact, for $r\in (\Z/N\Z)^\times$ we have the automorphism $\bar \varphi_r$ given by
$(a,b)\mapsto (ra,r^*b)$, where $r^*$ denotes the inverse of $r$ modulo $N$.

\begin{lemma}
\label{lem:act}
For $r\in (\Z/N\Z)^\times$ there exists a $\varphi_r\in \Orth(L)^+$ whose image under
\[
\Orth(L)^+\longrightarrow \Aut(A)
\]
restricts to the identity on $D'/D$ and to $\bar\varphi_r$ on $U(N)'/U(N)$.
The transformation $\varphi_r$ is uniquely determined up to multiplication by elements of $\Gamma(L)$.
\end{lemma}

\begin{proof}
It suffices to prove the assertion if $L=U(N)\oplus U$. We realize this lattice
as the group of integral matrices $X\in \Mat_2(\Z)$ whose left lower entry is divisible by $N$,
%\[
%L=\{\kabcd\in \Mat_2(\Z):\;N\mid c\}
%\]
with the quadratic form given by the determinant.
The group $\Gamma_0(N)\times \Gamma_0(N)$ acts on $L$ by $(\gamma_1,\gamma_2).X=\gamma_1 X\gamma_2^{-1}$ leaving the quadratic form fixed. This gives rise to a homomorphism to $\Orth(L)^+$. The subgroup $\Gamma_1(N)\times \Gamma_1(N)$ is mapped to $\Gamma(L)$.
We obtain a homomorphism
\[
(\Gamma_0(N)\times \Gamma_0(N))/(\Gamma_1(N)\times \Gamma_1(N))\longrightarrow \Aut(L'/L).
%\cong\Aut(U(N)'/U(N)).
\]
Using the fact that the left hand side is isomorphic to $(\Z/N\Z)^\times\times (\Z/N\Z)^\times$,
it is easily seen that $\bar\varphi_r$ is in the image of this map.
\end{proof}

\begin{theorem}
\label{thm:main}
Assume that $L\cong D\oplus U(N)\oplus U$ for some positive definite lattice $D$ of dimension $n-2$.
Then the map $\Lambda: S_{\kappa,L}\to \calH^{1,1}(X_\Gamma)$ is injective.
\end{theorem}

\begin{proof}
%1.
We put $A=L'/L$.
Let $g=\sum_{\mu\in A}\sum_m b(m,\mu)q^m\frake_\mu\in S_{\kappa,A}$ be an element in the kernel of $\Lambda$.
We denote by $g_\mu$ the components of $g$ with respect to the standard basis $(\frake_\mu)_{\mu\in A}$ of $\C[A]$.
We have to show that $g=0$.

1. We begin by noticing that $(\Z/N\Z)^\times$ acts on $S_{\kappa,A}$ via the automorphisms $\bar\varphi_r$, and it acts on $\calH^{1,1}(X_{\Gamma})$ via the transformations $\varphi_r$ for $r\in (\Z/N\Z)^\times$, see Lemma \ref{lem:act}. Moreover, in view of the third part of Theorem \ref{thm:lift}, the map $\Lambda$ is equivariant with respect to these actions.
Consequently, the action of $(\Z/N\Z)^\times$ preserves the kernel of $\Lambda$, and we obtain a decomposition of the kernel into isotypical components with respect to the characters of $(\Z/N\Z)^\times$. By orthogonality of characters, we may assume without loss of generality that
$g$ is contained in the $\chi$-isotypical component of $S_{\kappa,A}$ for some character
$\chi:(\Z/N\Z)^\times\to \C^\times$, that is,
\begin{align}
\label{eq:main1}
\bar\varphi_r .g=\chi(r)g, \quad r\in (\Z/N\Z)^\times.
\end{align}

2. To prove that $g=0$, we consider the Fourier expansion of $\Lambda(z,g)$.
Let $\ell,\ell'\in U\subset L$ be primitive isotropic vectors such that $(\ell,\ell')=1$. Then $U=\Z\ell+\Z\ell'$.
Let $V_0=V\cap \ell^\perp\cap \ell'{}^\perp$, and write $\calH$ for the corresponding tube domain realization of $\D$ as in \eqref{eq:td}.
Then $K=L\cap V_0\cong D\oplus U(N)$ and $L=K\oplus \Z\ell\oplus \Z\ell'$.
By means of the first part of Theorem \ref{thm:lift}, we see that for any $\lambda \in K'$ with $Q(\lambda)>0$,
we have
\begin{align}
\sum_{d\mid\lambda} d^{n-1}  b(Q(\lambda)/d^2, \lambda/d)=0.
\end{align}
Hence, by an inductive argument we find that
\begin{align}
\label{eq:main2}
b(Q(\lambda),\lambda)=0
\end{align}
for any $\lambda\in K'$ of positive norm. We
now show that this implies that {\em all\/} Fourier coefficients of $g$ vanish.

3. Let $e=\frac{1}{N}(0,1)\in U(N)'\subset L'$. This is a primitive isotropic vector of $U(N)'$
whose image in $A$ has order $N$. We use the newform theory developed in Section \ref{sect:3.1}
for the isotropic subgroup
\[
I_N =\langle e+L\rangle \subset A.
\]
For $d\mid N $ we also consider the  subgroup $I_d=\langle \frac{N}{d}e+L\rangle\subset A$.
We prove that all components $g_\mu$ vanish by induction on the number of prime divisors of
the content
\[
\cont_e(\mu)=\gcd(N(e,\mu),N)
\]
of $\mu$ with respect to $e$.

3.1. Let $\mu\in A$ with $\cont_e(\mu)=1$, that is, $(e,\mu) =\frac{r}{N}+\Z$ with $r\in (\Z/N\Z)^\times$. Using the action of $(\Z/N\Z)^\times$ on $g$ and \eqref{eq:main1}, we may assume without loss of generality  that $(e,\mu)= \frac{1}{N}+\Z$.
Then there exists a $\lambda\in \mu+K$ such that $(\lambda,e)=\frac{1}{N}$,
and for any $a\in \Z$ we have
\begin{align*}
\lambda +aNe\in \mu,\qquad Q(\lambda+aNe)=Q(\lambda)+a.
\end{align*}
But now \eqref{eq:main2} implies that the component $g_\mu$ vanishes identically.

3.2. Let $t>0$ and assume that $g_\mu=0$ for all $\mu\in A$ with $\Omega(\cont_e(\mu))<t$.
This means that $g$ is supported on
\begin{align*}
\bigcup_{\substack{d\mid N\\ \Omega(d)=t}} I_d^\perp  =\bigcup_{\substack{d\mid N\\ \Omega(d)=t}} \{\mu\in A:\;d\mid \cont_e(\mu) \}.
\end{align*}
According to Theorem \ref{thm:8}
there exist cusp forms $g_d\in S_{k,A(d)}$ such that
\begin{align}
\label{eq:main3}
g=\sum_{\substack{d\mid N\\ \Omega(d)\geq t}} g_d\uparrow_{I_d}^A,
\end{align}
where $A(d)=I_d^\perp/I_d$.

Let $\mu\in A$ with $\Omega(\cont_e(\mu))=t$ and put $d_0=\cont_e(\mu)$. There exists  a $r\in (\Z/N\Z)^\times$ such that
$(e,\mu)=\frac{r d_0}{N}+\Z$.
In view of \eqref{eq:main1}, we may assume that $(e,\mu)=\frac{d_0}{N}+\Z$.
The identity \eqref{eq:main3} and Remark \ref{rem:9} imply that
\[
g_\mu = (g_{d_0}\uparrow_{I_{d_0}}^A)_\mu
\]
and $g_{\mu+\mu'}=g_\mu$ for all $\mu'\in I_{d_0}$.
Since $(e,\mu)=\frac{d_0}{N}+\Z$, there exists a $\lambda\in \mu+K$ such that $(\lambda,e)=\frac{d_0}{N}$.
Moreover, for any $a\in \Z$ we have
\begin{align*}
\lambda +\frac{aN}{d_0}e\in \mu+I_{d_0},\qquad Q(\lambda+\frac{aN}{d_0}e)=Q(\lambda)+a.
\end{align*}
Now \eqref{eq:main2} implies that the component $g_\mu$ vanishes identically.
This shows that  $g_\mu=0$ for all $\mu\in A$ with $\Omega(\cont_e(\mu))=t$.
The theorem follows by induction.
\end{proof}

\begin{proof}[Proof of Theorem \ref{thm:mainintro}]
The assertion follows from Theorem \ref{lift+} by means of Theorem \ref{thm:main}.
\end{proof}

\begin{comment}
By means of Theorem \ref{lift+} we obtain the following converse theorem.

\begin{corollary}
\label{cor:main}
Assume that $L\cong D\oplus U(N)\oplus U$ for some positive definite
lattice $D$ of dimension $n-2\geq 1$, and write $A$ for its
discriminant group.  Then every meromorphic modular form $F$ with
respect to $\Gamma$ whose divisor is a linear combination of special
divisors as in \eqref{eq:divf}
%($c(\lambda)\in \Z$ with $c(\lambda)=c(-\lambda)$ and $c(\lambda)=0$ for all but finitely many $\lambda$),
is  (up to a non-zero constant factor)
the Borcherds lift $\Psi(z,f)$ of a weakly
holomorphic modular form $f\in M_{1-n/2,A^-}^{!}$.
% with integral principal part.
\end{corollary}

We also have the following slightly more general variant.
\end{comment}

\begin{proof}[Proof of Corollary \ref{cor:main2intro}]
Since $K$ is isotropic, there exists a primitive isotropic vector $\ell\in K$. Let $N_\ell\in \Z$ be a generator of the ideal $(\ell,K)\subset\Z$. It is easily seen that $N_\ell$ divides the level $N_K$ of $K$. We put $t=N_K/N_\ell$.
The sublattice
\[
K_0:=\{x\in K:\; (\ell,x)\in N_K\Z\}
\]
has index $t$ in $K$. Its dual is given by $K_0'=K'+\Z\frac{\ell}{N_K}$.
This implies that $K_0$ has also level $N_K$. It contains $\ell$ as a primitive isotropic vector and $(\ell,K_0)=N_K\Z$. Hence, according to Lemma  \ref{lem:6.1}, the lattice $K_0$ splits $U(N_K)$ as an orthogonal summand. Now the assertion follows from Theorem \ref{thm:mainintro}.
\end{proof}

\begin{comment}
\begin{corollary}
\label{cor:main2}
Assume that $L\cong K\oplus U$ for some isotropic lattice $K$ of signature $(n-1,1)$ where $n\geq 3$.
Let $M\subset K$ be a sublattice of finite index that splits a rescaled hyperbolic plane $U(N)$ over $\Z$ and put $B=M'/M$.
Then every meromorphic modular form $F$ with
respect to $\Gamma=\Gamma(L)$ whose divisor is a linear combination of special
divisors as in \eqref{eq:divf}
is  (up to a non-zero constant factor)
the Borcherds lift $\Psi(z,f)$ of some $f\in M_{1-n/2,B^-}^{!}$.
%h with integral principal part.
\end{corollary}

\begin{proof}
Since $M\oplus U \subset K\oplus U \cong L$, we may view $\Gamma(M\oplus U)$ as a subgroup of $\Gamma(L)$. Therefore $F$ is also a meromorphic modular form for $\Gamma(M\oplus U)$, and its divisor
is a linear combination of special divisors for $M\oplus U$. Hence the assertion follows from
Corollary \ref{cor:main}.
\end{proof}
\end{comment}

\begin{remark}
Note that the proof of Corollary \ref{cor:main2intro} gives an explicit construction of a sublattice $K_0\subset K$ as required.
% (which has index $t$).
%It also in view of Lemma \ref{lem:6.1} there always exists a sublattice $M\subset K$ as required in Corollary \ref{cor:main2}. Moreover, there exists such a sublattice $M$ which has the same level as $K$.
\end{remark}

\section{Lattices of prime level}

In this section we consider lattices of prime level. In particular, we prove a converse theorem for lattices of prime level that do not necessarily split a hyperbolic plane over $\Z$. We continue to use the notation of Section \ref{sect:4}.

\label{sect:6}
\subsection{Examples for which $\Lambda$ is not injective}
\label{sect:5.2}

For lattices $L$ that do not split a hyperbolic plane over $\Z$, the map $\Lambda:S_{\kappa,L}\to \calH^{1,1}(X_\Gamma)$ is not injective in general, since $\xi(N_{2-\kappa,L^-})$ can be non-trivial.
Here we give a direct construction of elements in the kernel for certain lattices.

Assume that $n\equiv 2\pmod{8}$, and let
$I\!I_{n,2}$ be the even unimodular lattice of signature $(n,2)$.
For a prime $p$, we
consider the lattice $L=I\!I_{n,2}(p)$
obtained by rescaling
%the quadratic form on $I\!I_{n,2}$
by $p$.
%Then
%\[
%L(p)\cong E_8(p)^{\frac{n-2}{8}}\oplus U(p)\oplus U(p).
%\]
%The corresponding discriminant group $A$ is a vector space over $\F_p$ of dimension $n+2$.
Then $A=L'/L\cong \F_p^{n+2}$ and $\sig(A)\equiv 0\pmod{8}$.
%We put $\kappa=\frac{n+2}{2}$.

For $M=\kabcd\in \Gl_2^+(\R)$ we define the Petersson slash operator in
integral weight $k$ on functions on $\H$ by
\begin{align}
(g\mid_k M)(\tau) = \det(M)^{k/2}(c\tau+d)^{-k} g(M\tau).
\end{align}
Hence, scalar matrices act trivially.
We denote by $W_p=\kzxz{0}{-1}{p}{0}$ the Fricke involution on the
space $M_k(\Gamma_0(p))$ of scalar valued modular forms of weight
$k$ for the group $\Gamma_0(p)$.
Recall that the Hecke operator $U_p$ acts on  $g=\sum_l a(l)q^l\in M_k(\Gamma_0(p))$ by
\begin{align}
g\mid U_p = \sum_l a(pl)q^l.
\end{align}

The restriction of $\rho_L$ to $\Gamma_0(p)$
acts trivially on the vector $\frake_0\in \C[A]$.
Hence, if $g\in
M_k(\Gamma_0(p))$, then
\begin{align}
\label{eq:lift}
\vec g = \sum_{\gamma\in \Gamma_0(p)\bs \Sl_2(\Z)}(g\mid_k \gamma) \rho_L^{-1}(\gamma)\frake_0
\end{align}
belongs to $M_{k,L}$. It is invariant under the action of $\Aut(A)$. Note that according to \cite[Corollary 5.5]{Sch}, every element of $M_{k,L}$ which is invariant under $\Aut(A)$, is the lift of a scalar valued form.
The Fourier expansion of $\vec g $ is computed (in greater generality) in \cite[Section 6]{Sch:Inv}.
If we write
%$g=\sum_l a(l)q^l$ and
$g\mid W_p = \sum_l \tilde a(l)q^l$, then for $\mu\in A$ and $m\in Q(\mu)+\Z$ the
$(m,\mu)$-th coefficient of $\vec g $
is given by
\begin{align}
\label{eq:liftfor}
\vec a (m,\mu) = \begin{cases}
p^{-k/2-n/2}\tilde a(pm),&\text{if $\mu\neq 0$,}\\[1ex]
a(m)+p^{-k/2-n/2} \tilde a(pm),&\text{if $\mu=0$.}
\end{cases}
\end{align}

\begin{proposition}
\label{prop:ker}
%Let $L=I\!I_{n,2}(p)$ be as above and put
Let $0\neq g\in S_\kappa(\Gamma_0(p))$, and assume that $g\mid U_p = -p^{\frac{\kappa}{2}-1} g\mid W_p$.
Then the corresponding vector valued form $\vec g\in S_{\kappa,L}$
does not vanish and $\Lambda(\vec g) =0$.
\end{proposition}

\begin{proof}
The proposition can be proved using the Fourier expansion of $\Lambda(\vec g)$ given in Theorem~\ref{thm:lift}.
We omit the details.
% See the version "converse3.tex"
\end{proof}

\begin{remark}
\label{rem:ker}
  Let $g\in S_k(\Gamma_0(p))$ be a newform with the property $g\mid W_p =\pm g$.
  Then, according to \cite[Theorem~9.27]{Kn}, we have $g\mid U_p =
  \mp p^{\frac{k}{2}-1} g$. Hence, there are many cusp forms
  satisfying the hypothesis of Proposition \ref{prop:ker}.
\end{remark}

\subsection{The converse theorem for lattices of prime level}

%We begin with some preparations.

Recall that for an isotropic vector $u\in V$ and $v\in V$ orthogonal to $u$, the Eichler element $E(u,v)\in \Orth(V)^+$ is defined by
\begin{align}
\label{eq:eichler}
E(u,v)(a)=a-(a,u)v+(a,v)u-Q(v)(a,u)u
\end{align}
for $a\in V$. It is easily seen that
if $u,v\in L$, then $E(u,v)\in \Gamma(L)$.

\begin{lemma}
\label{lem:eichler}
Let $L$ be an even lattice of level $N$.
Let $u\in L$ be an isotropic vector such that $(u,L)=N\Z$. If $v\in L'\cap u^\perp$, then $E(u,v)\in \Orth(L)^+$.
\end{lemma}

\begin{proof}
Since $L$ has level $N$, we have $NL'\subset L$.
Hence the assertion follows immediately from the definition \eqref{eq:eichler}.
\end{proof}

%For the rest of this section, let $L$ be an even lattice of prime
%level $p$ and signature $(n,2)$ with $n\geq 4$.   Put $A=L'/L$.  To
%prove a converse theorem in this case, the following proposition is
%vital.

\begin{proposition}
\label{prop:keyp}
Let $L$ be an even lattice of prime
level $p$ and signature $(n,2)$ with $n\geq 4$.
Let $g=\sum_{\mu\in A}\sum_m b(m,\mu)q^m\frake_\mu\in S_{\kappa,L}$ be an element in the kernel of $\Lambda$.
Let $u\in L$ be primitive isotropic, and assume $(u,L)=p\Z$. Then for every $v,\lambda\in L'\cap u^\perp$
%K':=L'\cap u^\perp/\Z \frac{u}{p}$
we have
\[
b(Q(\lambda),E(u,v)\lambda)=b(Q(\lambda),\lambda+(\lambda,v)u)= b(Q(\lambda),\lambda).
\]
\end{proposition}

%$\Lambda: S_{\kappa,L'/L}\to \calH^{1,1}(X_\Gamma)$ is injective.

\begin{proof} Put $A=L'/L$.
If $(v,\lambda)\in \Z$, then we have nothing to show. So we assume that $(v,\lambda)\in r/p+\Z$ with $r\in (\Z/p\Z)^\times$.
%In particular, $\bar v$ and $\bar \lambda$ are nontrivial in $A$.

We consider the Eichler transformation $E:=E(u,v)$, which belongs to $\Orth(L)^+$ according to Lemma \ref{lem:eichler}.
We have $E(\lambda)= \lambda+(\lambda,v)u$, and the image of $E$ generates a subgroup $G\subset\Aut(A)$ which is isomorphic to $\Z/p\Z$.
The group $G$ acts on $S_{\kappa,A}$ and on $\calH^{1,1}(X_\Gamma)$,
and in view of the third part of Theorem \ref{thm:lift}, the map $\Lambda$ is equivariant with respect to these actions.
%In view of the third part of Theorem \ref{thm:lift}, for every $a\in \Z$, the translate $E^a.g$ also belongs to the kernel of $\Lambda$.
Consequently, the action of $G$ preserves the kernel of $\Lambda$, and we obtain a decomposition of the kernel into isotypical components with respect to the characters of $G$.

For $s\in \Z/p\Z$ the $s$-isotypical component of $g$ is given by
\begin{align}
g_s=\sum_{a\;(p)} e(-as/p)E^a .g,
\end{align}
in particular, we have $E.g_s= e(s/p)g_s$.
If we write $g_s=\sum_{\mu\in A}\sum_m b_s(m,\mu)q^m\frake_\mu$, we have
\begin{align}
\label{eq:int1}
b_s(Q(\lambda),\lambda+(\lambda,v)u)= e(s/p)\cdot b_s(Q(\lambda),\lambda).
\end{align}
It suffices to show that $b_s(Q(\lambda),\lambda)=0$ for all $s\in (\Z/p\Z)^\times$.

Let $K=(L\cap u^\perp)/\Z u$.  We write the image of $\lambda$
%Since $(\lambda,v)\notin\Z$, the image of $\lambda$
in $K'\cong (L'\cap u^\perp)/\Z\frac{u}{p}$ as $d_0\lambda_0$, with a primitive vector  $\lambda_0\in K'$ and $d_0\in \Z_{>0}$. Since $(\lambda,v)\notin\Z$, the number $d_0$ is coprime to $p$.
We choose an auxiliary  prime $q$ coprime to $pd_0$,
and we put
$\lambda_1=qd_0\lambda_0\in K'$.
We employ Theorem \ref{thm:lift} (i) with $\ell=u$, to deduce that the $\lambda_1$-th Fourier coefficient of $\Lambda(g_s)$ vanishes, that is,
\[
\sum_{d\mid qd_0}d^{n-1}\sum_{a\;(p)} e(ad/p) b_s(Q(\lambda_1/d),\lambda_1/d+a\frac{u}{p}) =0,
\]
or equivalently,
\[
\sum_{d\mid qd_0}d^{n-1}\sum_{a\;(p)} e\big(a(\lambda_1,v)\big) b_s\big(Q(\lambda_1/d),\lambda_1/d+a(\lambda_1/d,v)u\big) =0.
\]
Using \eqref{eq:int1} and the fact that  $(\lambda_1,v)\equiv qr/p\pmod{\Z}$, we find
\[
\sum_{d\mid qd_0}d^{n-1}\sum_{a\;(p)} e\big(a(qr+s)/p\big) b_s\big(Q(\lambda_1/d),\lambda_1/d\big) =0.
\]
If $s\in (\Z/p\Z)^\times$ and  $qr\equiv -s\pmod{p}$, we obtain
\[
\sum_{d\mid qd_0}d^{n-1} b_s\big(Q(\lambda_1/d),\lambda_1/d\big) =0.
\]
If we split the sum over the divisors of $qd_0$ into a sum over the divisors coprime to $q$ and a sum over the divisors divisible by $q$, we obtain
\begin{align}
\label{eq:int2}
\sum_{d\mid d_0}d^{1-n}b_s\big(Q(qd\lambda_0),qd\lambda_0\big)
+q^{n-1}\sum_{d\mid d_0}d^{1-n}b_s\big(Q(d\lambda_0),d\lambda_0\big)=0.
\end{align}

By Dirichlet's theorem, there are infinitely many primes $q$ satisfying $qr\equiv -s\pmod{p}$.
If $q$ goes to infinity, then
the Weil bound for the coefficients of the cusp form $g_s$ of weight $\kappa=1+n/2$ implies that
for any $\eps>0$ we have
$
b_s(Q(q\lambda),q\lambda)=O(q^{\kappa-1/2+\eps})$.
Employing \eqref{eq:int2}, we obtain
\[
\sum_{d\mid d_0}d^{1-n}b_s\big(Q(d\lambda_0),d\lambda_0\big)=
- q^{1-n}\sum_{d\mid d_0}d^{1-n}b_s\big(Q(qd\lambda_0),qd\lambda_0\big)= O(q^{3/2-n/2+\eps}).
\]
By our assumption $n>3$, the right hand side goes to zero as $q\to \infty$, and therefore
\[
\sum_{d\mid d_0}d^{1-n}b_s\big(Q(d\lambda_0),d\lambda_0\big)=0.
\]
An inductive argument now shows that $b_s\big(Q(d\lambda_0),d\lambda_0\big)=0$ for all $s\in (\Z/p\Z)^\times$ and
all $d\mid d_0$. This proves the assertion.
%
\begin{comment}
By means of \eqref{eq:int1}, we obtain
\[
\sum_{a\;(p)} e\left(\frac{a(qr+s)}{p}\right)b_s(Q(q\lambda),q\lambda) + q^{n-1}\sum_{a\;(p)} e\left(\frac{a(qr+s)}{p}\right)b_s(Q(\lambda),\lambda)=0.
\]
If $s\in (\Z/p\Z)^\times$ and  $qr\equiv -s\pmod{p}$, we find that
\begin{align}
\label{eq:int2}
b_s(Q(q\lambda),q\lambda) + q^{n-1}b_s(Q(\lambda),\lambda)=0.
\end{align}
By Dirichlet's theorem, there are infinitely many such primes $q$. If $q$ goes to infinity, the the Weil bound for the coefficients of the cups form $g_s$ of weight $\kappa=1+n/2$ implies that
\[
b_s(Q(q\lambda),q\lambda)=O(q^{\kappa-1/2}).
\]
Using \eqref{eq:int2}, we obtain
\[
b_s(Q(\lambda),\lambda)=  -q^{1-n}b_s(Q(q\lambda),q\lambda)= O(q^{3/2-n/2}).
\]
By our assumption $n>3$, the right hand side goes to zero as $q\to \infty$, and therefore
$b_s(Q(\lambda),\lambda)=0$.
\end{comment}
\end{proof}

Next we show that $g$ as in the previous proposition behaves nicely
under the action of $\Aut(A)$.  We first introduce some notation for
the rest of this section. Let $L$ be an even lattice of prime level
$p$ and signature $(n,2)$ with $n\geq 4$.  We also assume that $L$ has
Witt rank 2, which is automatically true if $n>4$.

In view of Lemma \ref{lem:6.1}, possibly replacing $L$ by a sublattice
of level $p$, we may assume that $L=D\oplus M$, where $D$ is a
positive definite even lattice of level $p$ and rank $n-2$, and
$M\cong U(p)\oplus U(p)$. We identify $M$ with the lattice of integral
$2\times 2$ matrices with the quadratic form $Q(X)=p\det(X)$. The
group $\Gamma(1)\times\Gamma(1)$ acts on $M$ by orthogonal
transformations via
$(\gamma_1,\gamma_2).X=\gamma_1 X\gamma_2^{-1}$. The action gives rise to a homomorphism
$\Gamma(1)\times\Gamma(1)\to \Orth(M)^+$ whose kernel is $\{\pm 1\}$.
We write $\lambda\in L$ as $\lambda=\lambda_D+\lambda_M$ with
$\lambda_D\in D$ and $\lambda_M\in M$.  Moreover, we denote the
canonical projection $L'\to A\cong D'/D\oplus M'/M$ by $\lambda\mapsto
\bar \lambda$.  We let $K\subset L$ be the sublattice of those
$\lambda\in L$ for which $\lambda_M$ is a diagonal matrix. Hence
$K\cong D\oplus U(p)$.
\label{splitL}

We say that $\mu\in A$ has {\em normal form} if $\mu_D = 0$ in $D'/D$ and
\[
\mu_M = \begin{cases} \kzxz{0}{0}{0}{0},&\text{if $\mu = 0$,}\\
\kzxz{1/p}{0}{0}{Q(\mu)},&\text{if $\mu \neq 0$.}
\end{cases}
\]
If we apply Witt's theorem for the discriminant form $A$, we see that in every $\Aut(A)$-orbit of $A$ there exists a unique element in normal form. It only depends on the order of $\mu$ in $A$ and on $Q(\mu)\in \Q/\Z$.
Let $r$ be the rank of the $\F_p$-vector space $A$. The splitting $L=D\oplus M$ implies that $r\geq 4$.
%Finally, we let $\chi_A$ be the quadratic character associated to $A$ as in \eqref{eq:defchia}.

\begin{proposition}
\label{prop:normform}
%For every $\lambda\in L'$ there exists a $\lambda_1\in L'$ of the same norm such that $\bar\lambda_1$ has normal form and
Let $g$ be as in Proposition \ref{prop:keyp}.
For every $\lambda\in L'$ there exists a $\gamma\in \Orth(L)^+$ such that $\gamma\bar\lambda$ has normal form and
%\in M'/M$ and such that $\lambda_1\in L'$ of the same norm such that $\bar\lambda_1$ has normal form and
\[
b(Q(\lambda),\lambda)= b(Q(\lambda),\gamma\lambda).
\]
\end{proposition}

\begin{proof}
%\texttt{Put in! Use Proposition \ref{prop:keyp}.}
1. We first show that there exists a $\gamma\in \Orth(L)^+$ such that $\gamma\bar\lambda\in M'/M$ and such that
\[
b(Q(\lambda),\lambda)= b(Q(\lambda),\gamma\lambda).
\]
%In fact, if $\bar\lambda_D\neq 0$, then
%
%We may assume that $\bar\lambda_D\neq 0$.
%If $\bar\lambda_D=0$ in $D'/D$ we can take $\gamma=1$.

1.1. We begin by looking at the case that  $\bar\lambda_D\neq 0$ and $\bar\lambda_M\neq 0$.
By the elementary divisor theorem for $\Gamma(1)$, we may choose a basis of $M$ such that the coordinate vector of $\lambda_M$ has the form
\[
\lambda_M=\frac{1}{p}\zxz{a}{0}{0}{d}
\]
with $a\in \Z_{>0}$ and $d\in \Z$ divisible by $a$.

We choose a $v_D\in D'$ such that $(\lambda_D,v_D)a\equiv -\frac{1}{p}\pmod{\Z}$. This is possible, since $\bar\lambda_D\neq 0$ and $(a,p)=1$. Then we define vectors
\begin{align*}
u&=pa\lambda_D -\zxz{0}{1}{pa^2Q(\lambda_D)}{2pQ(\lambda_D)}\in L,\\
v&=v_D-\zxz{0}{0}{a(\lambda_D,v_D)}{0}\in L'.
\end{align*}
It is easily checked that $u\in L$ is primitive isotropic and $(u,\lambda)=(u,v)=0$.
The Eichler element $E=E(u,v)$ belongs to $\Orth(L)^+$, and we have
\begin{align*}
E\lambda&= \lambda+(\lambda,v)u\\
&=\lambda+(\lambda_D,v_D)u\\
%=\lambda+(\lambda_D,v_D)pa\lambda_D.
&\equiv\lambda_M+(\lambda_D,v_D)u_M \pmod{L}.
\end{align*}
Now the claim follows from Proposition \ref{prop:keyp}.

1.2. In the case that  $\bar\lambda_D\neq 0$ and $\bar\lambda_M= 0$ we put $u=\kzxz{0}{1}{0}{0}$ and
choose a $v\in D'$ such that $(\lambda_D,v)\notin\Z$. Then
$\lambda_1:=E(u,v)\lambda= \lambda+(\lambda_D,v)u$ has the property that
$\bar\lambda_D\neq 0$ and $\bar\lambda_M\neq 0$. By Proposition \ref{prop:keyp}, we have
$b(Q(\lambda),\lambda)=b(Q(\lambda_1),\lambda_1)$. Now we can argue as in case 1.1.

2. We may now assume that $\bar\lambda_D =0$. We show that there is a $\gamma\in \Orth(M)^+\subset \Orth(L)^+$ as requested. By the elementary divisor theorem there exists a $\gamma\in(\Gamma(1)\times\Gamma(1))/\{\pm 1\}\subset \Orth(M)^+$ such that
$\gamma . \bar \lambda$ has normal form.
Hence it suffices to show that for all $\gamma\in \Gamma(1)\times\Gamma(1)$ we have
\[
b(Q(\lambda),\lambda)= b(Q(\lambda),\gamma.\lambda).
\]
It suffices to prove this for the generators
$(\kzxz{1}{1}{0}{1},1)$,  $(\kzxz{1}{0}{1}{1},1)$, $(1,\kzxz{1}{1}{0}{1})$, $(1,\kzxz{1}{0}{1}{1})$. We illustrate the argument with the first generator $\gamma_1$, for the others it is analogous.
If we write $\lambda_M=\frac{1}{p}\kabcd$ and $u=\kzxz{c}{d}{0}{0}$, we have
\[
%(\zxz{1}{1}{0}{1},1)
\gamma_1.\lambda_M = \frac{1}{p}\zxz{a+c}{b+d}{c}{d}= \lambda_M+\frac{1}{p}u.
\]
If $p\mid (c,d)$, then there is nothing to prove. If $p\nmid (c,d)$ we choose $\alpha,\beta\in \Z$ such that
$\alpha d-\beta c\equiv 1\pmod{p}$ and put
\[
v=\frac{1}{p} \zxz{\alpha}{\beta}{0}{0}.
\]
Then $(\lambda,u)=(v,u)=0$ and $E(u,v)\lambda = \lambda + (\lambda,v)u= \lambda+u/p$.
Again, the claim follows from Proposition \ref{prop:keyp}.
\end{proof}

\begin{remark}
\label{rem:inv0}
The argument of Proposition \ref{prop:normform} also shows that $\Orth(L)^+$ acts transitively on the $\Aut(A)$-orbits of $A$.
%for every $\mu\in A$ there exists a $\gamma\in \Orth(L)^+$ such that $\gamma \mu$ has normal form.
\end{remark}

\begin{corollary}
\label{cor:inv0}
Let $g\in S_{\kappa,L}$ with Fourier coefficients $b(m,\mu)$, and assume $\Lambda(g)=0$.
% be an element in the kernel of $\Lambda$.

i) For every $\lambda\in L'$ and for every $\gamma\in \Aut(L'/L)$ we have
$
b(Q(\lambda),\lambda)= b(Q(\lambda),\gamma\lambda)$.

ii) If $g$ is actually contained in $S_{\kappa,L}^+$,
% and $\Lambda(g)=0$,
then $g$ is invariant under $\Aut(L'/L)$.
\end{corollary}

\begin{proof}
i)
In every $\Aut(A)$-orbit of $A$ there exists a unique element in normal form. Hence the corollary directly follows from Proposition \ref{prop:keyp} and Proposition \ref{prop:normform}.

ii) Let $\mu\in A$ and $m\in \Z+Q(\mu)$ be positive.
If there does not exist any  $\lambda\in \mu+L$ for which $Q(\lambda)=m$, then $Z(m,\mu)=0\in \Div(X_\Gamma)$. Hence the harmonic Maass form $f_{m,\mu}\in H_{2-\kappa,L^-}$ with principal part $\frac{1}{2}q^{-m}( \frake_\mu+\frake_{-\mu})$ belongs to $N_{2-\kappa,L^-}$. Since $g\in  S_{\kappa,L}^+$, we have
\[
b(m,\mu) = \{f_{m,\mu},g\}= (\xi(f_{m,\mu}),g)=0.
\]
Moreover, Remark \ref{rem:inv0} implies that for every $\gamma\in \Aut(A)$ there does not exists any $\lambda\in \gamma\mu+L$ for which $Q(\lambda)=m$. Consequently, we have $b(m, \gamma\mu)=0$ as well.

Combining this with (i) we find that  $g$ is invariant under $\Aut(L'/L)$.
%On the other hand, if there exists a $\lambda\in \mu+L$ for which $Q(\lambda)=m$, then (i) states that $b(m,\mu)= b(m,\gamma\mu)$ for every $\gamma\in \Aut(A)$.
%Therefore $g$ is invariant under $\Aut(L'/L)$.
\end{proof}

%\begin{corollary}
%\label{cor:inv}
%
%ii) There exists a $h\in S_\kappa(p,\chi_A)$ such that $g=\vec h$ as in Proposition \ref{prop:fevector}.
%\end{corollary}

%Note that according to \cite[Corollary 5.5]{Sch}, every element of $S_{\kappa,A}$ which is invariant under $\Aut(A)$, is the lift of a scalar valued form for $\Gamma_0(p)$ similarly as in \eqref{eq:lift}.

%Now we can use \cite{Sch}, to deduce that $g$ is the lift as in \eqref{eq:lift} of some scalar valued form
%$h\in S_\kappa(\Gamma_0(p),\chi_A)$. Employing the Fourier expansion of $\Lambda$ we find some conditions on $h$ that %allow us to prove a converse theorem.

Since $L$ has level $p$, the quotient $L/pL'$ is an $\F_p$-vector space of dimension $n+2-r$. The quadratic form $Q$ on $L$ induces a non-degenerate $\F_p$-valued quadratic form on $L/pL'$.

\begin{lemma}
\label{lem:rep}
If $L/pL'$ represents $0\in \F_p$ non-trivially,
then for any $m\in p\Z$ there exists a $\lambda\in L$ such that $Q(\lambda)=m$ and $\lambda/p\notin L'$.
Moreover, such a $\lambda$ can be chosen primitively in $L'$.
\end{lemma}

\begin{proof}
The hypothesis implies that there is a
$\lambda_0\in L$ such that $m_0:=Q(\lambda_0)\in p\Z$ and $\lambda_0/p\notin L'$.
We write $L=D\oplus M$ as on page \pageref{splitL}. Acting with $\Gamma(1)\times \Gamma(1)\subset \Orth(L)^+$ we may assume that
\[
\lambda_0=\lambda_{0D}+\zxz{a}{0}{0}{b}
\]
with $\lambda_{0D}\in D$ and $a,b\in \Z$. Then, for $t\in \Z$, the vector
\[
\lambda=\lambda_{0D}+\zxz{a}{1}{t}{b}
\]
in $L$ represents $m_0-pt$. It is primitive in $L'$.
\end{proof}

\begin{theorem}
\label{thm:mainp}
If  $g\in S_{\kappa,L}^+$ and $\Lambda(g)=0$, then $g=0$.
\end{theorem}

\begin{proof}
We first consider the case in which $L/pL'$ represents $0\in \F_p$ non-trivially.
%in which there exists a $\lambda_0\in L$ such that $p\mid Q(\lambda_0)$ and $\lambda_0/p\notin L'$.
Then, according to Lemma \ref{lem:rep}, for any $m\in p\Z$ there exists a $\lambda\in L$ which is primitive in $L'$ such that $Q(\lambda)=m$.
We note that for such $m$ we have
\begin{align}
\label{thm:mt1}
b(m,0)=b(m,\ell/p).
\end{align}
In fact, using Corollary \ref{cor:inv0} and the action of $\Gamma(1)\times \Gamma(1)\subset \Orth(L)^+$, we may assume that $\lambda$ is actually contained in $K$ and primitive in $K'$.
Now the claim follows from the vanishing of the $\lambda$-th coefficient of $\Lambda(g)$ and the formula for the Fourier expansion  given in Theorem~\ref{thm:lift}.

Next we deduce that for any $m\in \frac{1}{p}\Z$ and any $\mu\in A\setminus\{0\}$ the Fourier coefficient $b(m,\mu)$ of $g$ vanishes. In fact, if there is no $\lambda\in L'$ such that $Q(\lambda)=m$ and $\bar\lambda =\mu$, then $Z(m,\mu)=0\in \Div(X_\Gamma)$. Since $g\in S_{\kappa,L}^+$, this implies $b(m,\mu)=0$.
On the other hand, if there exists a $\lambda\in L'$ such that $Q(\lambda)=m$ and $\bar\lambda =\mu$, then $\mu\neq 0$ implies that $\lambda/p\notin L'$. As in the proof of Lemma \ref{lem:rep}, using the action of $\Gamma(1)\times \Gamma(1)\subset \Orth(L)^+$, we may assume that $\lambda$ is primitive in $K'$.
%Moreover, using the action of $\Gamma(1)\times \Gamma(1)\subset \Orth(L)^+$, we may assume that $\lambda\in K'$ primitive.
Employing the Fourier expansion of $\Lambda(g)$ given in Theorem \ref{thm:lift}, and the vanishing of
the $p\lambda$-th coefficient, we find that
\begin{align}
\label{thm:mt2}
0=b(Q(p\lambda),0)-b(Q(p\lambda),\ell/p) + p^n b(Q(\lambda),\lambda).
\end{align}
Combining this with \eqref{thm:mt1}, we obtain that $b(Q(\lambda),\lambda)=b(m,\mu)$ vanishes.

Hence $g$ is supported on its $0$-th component. This is not possible by Proposition \ref{prop:3}.

Now we consider the case in which $L/pL'$ does not represent $0$ non-trivially.
%there exists no $\lambda_0\in L$ such that $p\mid Q(\lambda_0)$ and $\lambda_0/p\notin L'$.
Then, according to Lemma \ref{lem:rep}, for any $\lambda \in L$  with $p\mid Q(\lambda)$, we have $\lambda/p\in L'$. This implies that for any $m \in \frac{1}{p}\Z$ we have
\[
Z(p^2m,0)=\sum_{\mu \in A} Z(m,\mu)\in \Div(X_\Gamma).
\]
Since $g\in S_{\kappa,L}^+$, we get for the Fourier coefficients the corresponding relation
\begin{align}
\label{thm:mt3}
b(p^2m,0)=\sum_{\mu \in A} b(m,\mu).
\end{align}

For $m\in \frac{1}{p}\Z$ we put
\[
B(m)=\begin{cases} b(m,\mu),&\text{if there exists a $\mu\in A\setminus\{0\}$ such that $Q(\mu)\equiv m\pmod{\Z}$,}\\
0,&\text{otherwise.}
\end{cases}
\]
Because of Corollary \ref{cor:inv0} this definition is independent of the choice of $\mu$.

Let $m\in \frac{1}{p}\Z$ with $\ord_p(m)\leq 0$. We claim that for every $r\in \Z_{\geq 0}$ there are integers
$C_m(r)\geq p^{rn}$ and $C'_m(r)\geq 0$ such that
\begin{align}
\label{thm:mt4}
B(p^{2r}m)&=C_m(r) B(m),\\
\label{thm:mt5}
b(p^{2r}m,0)&= C'_m(r) B(m).
\end{align}
We prove this claim by induction on $r$.

If $r=0$, then for \eqref{thm:mt4} we have nothing to show.
For \eqref{thm:mt5} we note that if $\ord_p(m)=-1$ then there is no $\lambda \in L$ such that $Q(\lambda)=m$. Hence $Z(m,0)=0\in \Div(X_\Gamma)$ and therefore $b(m,0)=0$. If $\ord_p(m)=0$ and there is no $\lambda \in L$ such that $Q(\lambda)=m$, we find that $b(m,0)=0$ for the same reason. On the other hand, if there is a $\lambda \in L$ such that $Q(\lambda)=m$, then $\lambda/p\notin L'$. Consequently, we may argue as in \eqref{thm:mt1} to see that
$b(m,0)=B(m)$.

If $r>0$, we obtain \eqref{thm:mt5} directly from \eqref{thm:mt3} and the induction assumption.
To obtain \eqref{thm:mt4}, we take a
$\lambda\in K'\setminus K$ which is primitive in $K'$ with the property $Q(\lambda)=p^{2(r-1)}m$.
Such a vector exists, since $K$ contains $U(p)$ as a direct summand.
The formula for the $p\lambda$-th Fourier coefficient of $\Lambda(g)$ implies that
\begin{align}
\label{thm:mt6}
0=b(p^{2r}m,0)-B(p^{2r}m)+p^{n} B(p^{2(r-1)}m).
\end{align}
Using the induction assumption and \eqref{thm:mt5}, we find
\begin{align*}
B(p^{2r}m)=C'_m(r)B(m)+p^{n} C_m(r-1) B(m).
\end{align*}
This proves the claim.

Now we compare \eqref{thm:mt4} with the Weil estimate for the growth of Fourier coefficients of cusp forms of weight $\kappa=1+n/2$. It implies that for any $\eps>0$ we have
\[
B(p^{2r}m)=O_\eps(p^{2r(\kappa/2-1/4+\eps)}),\quad  r\to \infty.
\]
Since $n>1$, the assumption $B(m)\neq 0$ leads to a contradiction.
We conclude that $B(m)=0$. Varying $m$ and employing \eqref{thm:mt4} and \eqref{thm:mt5}, we find that $g=0$.
\end{proof}

\begin{proof}[Proof of Theorem \ref{thm:main2intro}]
The assertion follows from Theorem \ref{lift+} using Theorem \ref{thm:mainp}.
\end{proof}

\section{Applications}

\label{sect:7}

\subsection{Ranks of Picard groups}

\label{sect:7.1}

%[Improve the results of \cite{Br2}.]
Throughout this section we assume that $n\geq 2$ and that $n$ is greater than the Witt rank of $L$.
When $\Gamma$ acts freely on $\D$, we define the Picard group $\Pic(X_\Gamma)$
%of $X_\Gamma$
to be the group of isomorphism classes of algebraic line bundles on $X_\Gamma$. In general, we choose a normal subgroup $\Gamma'\subset \Gamma$ of finite index which acts freely on $\D$ and define $\Pic(X_\Gamma):= \Pic(X_{\Gamma'})^{\Gamma/\Gamma'}$.
%We briefly write $\Pic(X_\Gamma)_\C:=\Pic(X_\Gamma)\otimes_\Z\C$.
This definition is independent of the choice of $\Gamma'$, and our assumption on $n$ implies that $\Pic(X_\Gamma)$ is a finitely generated abelian group.
We write $\Psp(X_\Gamma)$ for the subgroup generated by the special divisors $Z(m,\mu)$.

The computation of the Picard group is a difficult problem. For certain interesting modular varieties of low dimension it was determined using algebraic geometric methods (se e.g. \cite{GeNy}, \cite{FS}, \cite{CFS}), but little is known in general. Here we use our injectivity results for the map $\Lambda^+$ to
 describe $\Psp(X_\Gamma)$ explicitly. In particular we obtain a
 lower bound for the rank of $\Pic(X_\Gamma)$ improving \cite{Br2}.

According to \cite[Theorem 5.9]{Br1}, we have a commutative diagram
\begin{align}
\label{eq:diag1}
\xymatrix{
S_{\kappa,L} \ar[r]\ar[dr]_\Lambda & \Psp(X_\Gamma)\otimes_\Z\C \ar[d]\\
    & \calH^{1,1}(X_\Gamma).
}
\end{align}
Here the horizontal map is defined by taking $g\in S_{\kappa,L}$ to the line bundle corresponding to
$Z(f)$, where $f\in H_{1-n/2,L^-}$ with vanishing constant term $c^+(0,0)$ such that $\xi(f)=g$.
The vertical map is given by mapping the line bundle corresponding to
$Z(f)$ to $\Lambda(g,z)$.

\begin{corollary}
\label{cor:rankest}
If $\Lambda^+: S_{\kappa,L}^+\to \calH^{1,1}(X_\Gamma)$ is injective, then
\[
\rank(\Psp(X_\Gamma))= 1+ \dim(S_{\kappa,L}^+).
\]
\end{corollary}

\begin{proof}
If $\Lambda^+$ is injective, then the diagram \eqref{eq:diag1} implies that the map $S_{\kappa,L}^+\to  \Psp(X_\Gamma)\otimes_\Z\C$ is injective, too. The image of this map is the codimension $1$ subspace generated by the classes of line bundles of degree $0$. On the other hand, the class of the canonical bundle has non-vanishing degree and therefore generates a one-dimensional subspace of $\Pic(X_\Gamma)\otimes_\Z\C$ which is not contained in the image of the map. It is contained in $\Psp(X_\Gamma)\otimes_\Z\C$, since there are always Borcherds products of nonzero weight.
\end{proof}

\begin{comment}
\begin{theorem}
\label{thm:rankest}
%Assume that $L\cong D\oplus U(N_1)\oplus U(N_2)$ for some positive definite lattice $D$ of dimension $n-2>0$ and positive integers $N_1,N_2$. Let $L_0$ be the lattice $L\cong D\oplus U(N_1)\oplus U$ and put $A_0=L_0'/L_0$.
Assume that $L\cong D\oplus U(N)\oplus U$ for some positive definite even lattice $D$ of dimension $n-2>0$ and some positive integer $N$. Then
\[
\rank(\Psp(X_\Gamma))= 1+ \dim(S_{\kappa,L}).
\]
\end{theorem}
\end{comment}

\begin{proof}[Proof of Theorem \ref{thm:rankestintro}]
%[Proof of Theorem \ref{thm:rankest}]
Since $L$ splits a hyperbolic plane over $\Z$, we have $S_{\kappa,L}^+=S_{\kappa,L}$. The map $\Lambda$ is injective by Theorem \ref{thm:main}. Hence, the assertion follows from Corollary \ref{cor:rankest}.
\end{proof}

\begin{remark}
The dimension of $S_{\kappa,L}$ can be explicitly computed by means of the Selberg trace formula or the Riemann-Roch theorem,
% see Section \ref{sect:dimfor}.
see \cite[p.~228]{Bo3}.
Moreover, it can be estimated by means of \cite[Theorem 6]{Br2}.
\end{remark}

\begin{example}
As an example we consider the lattice $L=\Z\oplus U(N)\oplus U$, where $\Z$ is equipped with the even quadratic form $x\mapsto x^2$. Then $L$ has signature $(3,2)$, Gram determinant $4N^2$,  and $L'/L\cong
%the discriminant group is isomorphic to
\Z/4\Z\oplus (\Z/N\Z)^2$.
Let $\Gamma=\Gamma(L)$ be the discriminant kernel subgroup of $\Orth(L)^+$. Using the isomorphism between the Spin group of $L\otimes_\Z\Q$ and the symplectic group $\Symp_2(\Q)$, we may view $\Gamma$ as a congruence subgroup of the Siegel modular group $\Symp_2(\Z)$ of genus two.
%Valery writes: I did not check this carefully but one gets the stable  group as the image of
%\Gamma_1(N)={
%AB
%CD,
%C=0 mod N, det A=det D=1 mod N.
In Table \ref{table:1} we list the ranks of $\Psp(X_\Gamma)$ for $N< 20$.

\begin{table}[h]
\caption{\label{table:1} Ranks of Picard groups}
\begin{tabular}{c|rrrrrrrrrrrrrrrrrrr }
%\hline
\rule[-3mm]{0mm}{8mm}
$N$ & 1& 2 &3& 4& 5& 6& 7& 8& 9& 10& 11& 12& 13& 14& 15& 16& 17& 18& 19\\
\hline %\rule[-3mm]{0mm}{8mm}
\rule[-3mm]{0mm}{8mm}
$\rank(\Psp(X_\Gamma))$ & 1& 1& 1& 1& 3& 2& 4& 3& 7& 9& 11& 7& 19& 16& 19& 17& 33& 28& 37\\
%\rule[-3mm]{0mm}{8mm}
%$\rank(\Psp(X_\Gamma))$ & 0& 0& 0& 0& 2& 1& 3& 2& 6& 8& 10& 6& 18& 15& 18& 16& 32& 27& 36\\
%
% [0, 0, 0, 0, 2, 1, 3, 2, 6, 8, 10, 6, 18, 15, 18, 16, 32, 27, 36]
% \hline
\end{tabular}
\end{table}
\end{example}

\subsection{Other theta lifts and holomorphic differential forms}
\label{sect:7.2}

There are variants of the injectivity results of the present paper for other theta lifts of vector valued elliptic modular forms.
As an example we consider Borcherds' description (and generalization) of the liftings of Maass, Gritsenko, and Doi-Naganuma.

Under our assumption on $n$, Theorem 14.3 of \cite{Bo2} implies that for any integer $k>1$, there is a theta lift
\[
\vartheta_k:
S_{1-n/2+k,L^-}\longrightarrow S_k(\Gamma)
\]
to cusp forms of weight $k$ for the group $\Gamma$. The Fourier expansion of this lift is very similar
to the one of $\Lambda$ given in Theorem \ref{thm:lift}.
A straightforward adaption of the proof of Theorem~\ref{thm:main} yields the following result.

\begin{theorem}
\label{thm:mainvar}
If $L\cong D\oplus U(N)\oplus U$ for some positive definite even lattice $D$ of dimension $n-2>0$ and some positive integer $N$, then $\vartheta_k$ is injective.
\end{theorem}

\begin{corollary}
Let $L$ be as in Theorem \ref{thm:mainvar}, and let $\calH^{n,0}(X_\Gamma)$ be the space of square-integrable holomorphic $n$-forms on $X_\Gamma$. Then we have the lower bound
\[
\dim(\calH^{n,0}(X_\Gamma))\geq \dim(S_{1+n/2,L^-}).
\]
\end{corollary}

\begin{proof}
According to \cite[Lemma 5.10]{Br1}, the space $\calH^{n,0}(X_\Gamma)$ is isomorphic to $S_n(\Gamma)$. Therefore, the assertion follows from the injectivity of $\vartheta_n$.
\end{proof}

Note that for lattices that do not split a hyperbolic plane over $\Z$, the map $\vartheta_k$ can have a non-trivial kernel. For instance, if $p$ is a prime and $L=I\!I_{n,2}(p)$ as in Section \ref{sect:5.2}, we have the following result:

\begin{proposition}
\label{prop:ker2}
%Let $L=I\!I_{n,2}(p)$ be as above and put
Let $0\neq g\in S_{1-n/2+k}(\Gamma_0(p))$, and assume $g\mid U_p = -p^{k/2-n/4-1/2} g\mid W_p$.
Then the corresponding vector valued form $\vec g\in S_{1-n/2+k,L^-}$
does not vanish and $\vartheta_k(\vec g) =0$.
\end{proposition}

This can be proved by means of the Fourier
expansion of $\vartheta_k$
%, see \cite[Theorem 14.3]{Bo2},
and
\eqref{eq:liftfor}. In view of Remark~\ref{rem:ker}, there exist many
cusp forms $g$ satisfying the hypothesis of the proposition.


\begin{thebibliography}{GeNy}

\bibitem[Ba]{Ba}  \emph{A. Barnard},
The Singular Theta Correspondence, Lorentzian Lattices and Borcherds-Kac-Moody Algebras,
Ph.D. Dissertation, U.C. Berkeley (2003).
arXiv:math/0307102v1 [math.GR]



\bibitem[Bo1]{Bo1}  \emph{R. E. Borcherds}, Automorphic forms on $\Orth_{s+2,2}(\R)$
and infinite products, Invent. Math. {\bf 120} (1995), 161--213.

\bibitem[Bo2]{Bo2}
\emph{R. E. Borcherds}, Automorphic forms with singularities on
Grassmannians, Invent. Math. \textbf{132} (1998), 491--562.

\bibitem[Bo3]{Bo3}
 \emph{R. Borcherds}, The Gross-Kohnen-Zagier theorem in higher
dimensions, Duke Math. J. \textbf{97} (1999), 219--233.
%Correction in: Duke Math J. \textbf{105} No. 1 p.183--184.

%\bibitem[Bo4]{Bo4} {\em R. E. Borcherds}, Reflection groups of
%  Lorentzian lattices, Duke Math. J. {\bf 104} (2000), 319--366.

%\bibitem{Br}
%J. Bruinier,
%\emph{Borcherds products and Chern classes of Hirzebruch-Zagier divisors},
%Inv. Math. \textbf{138} (1999), 51-83.

\bibitem[Br1]{Br1} \emph{J. H. Bruinier}, Borcherds products on
  $\Orth(2,l)$ and Chern classes of Heegner divisors, Springer Lecture
  Notes in Mathematics {\bf 1780}, Springer-Verlag (2002).

\bibitem[Br2]{Br2} \emph{J. H. Bruinier}, On the rank of Picard groups of modular varieties attached to orthogonal groups, Compos. Math. {\bf 133} (2002), 49-63.

\bibitem[Br3]{Br3} \emph{J. H. Bruinier}, Regularized theta lifts for orthogonal groups over totally real fields, J. Reine Angew. Math., accepted for publication.

%\bibitem[BBK]{BBK} {\em J. H. Bruinier, J. Burgos, and U. K\"uhn}, Borcherds products and
%arithmetic intersection theory on Hilbert modular surfaces, Duke
%Math. J. {\bf 139} (2007), 1--88.
%\texttt{math.NT/0310201}


\bibitem[BrFr]{BrFr} {\em J. H. Bruinier and E. Freitag}, Local Borcherds products,
Annales de l'Institut Fourier {\bf 51.1} (2001), 1--26.

\bibitem[BF1]{BF} {\em J. H. Bruinier and J. Funke}, On two geometric
  theta lifts, Duke Math. Journal. {\bf 125} (2004), 45--90.

\bibitem[BF2]{BF2} {\em J. H. Bruinier and J. Funke}, On the
  injectivity of the Kudla-Millson lift and surjectivity of the
  Borcherds lift. In: Moonshine -- the first quarter
  century and beyond, Eds.: J. Lepowski, J. McKay, M. Tuite, Cambridge
  University Press (2010), 12--39.

%\bibitem[BK]{BK} \emph{J. H. Bruinier and U. K\"uhn},
%Integrals of automorphic Green's functions associated to Heegner
%divisors, Int. Math. Res. Not. {\bf 2003:31} (2003), 1687--1729.

\bibitem[BY]{BY} \emph{J. H. Bruinier and T. Yang}, Faltings heights of CM cycles and derivatives of $L$-functions, Invent. Math. {\bf 177} (2009), 631--681.

%\bibitem[BKK]{BKK} {\em J. Burgos, J. Kramer, and U. K\"uhn},
%Cohomological arithmetic Chow groups,  J. Inst. Math. Jussieu.
%{\bf 6}, 1--178 (2007).

\bibitem[BO]{BO} \emph{J. H. Bruinier and K. Ono}, Heegner divisors, $L$-functions and harmonic weak Maass forms, Annals of Math. {\bf 172} (2010), 2135--2181.

\bibitem[CFS]{CFS} {\em S. Cynck, E. Freitag and R. Salvati Manni},
The geometry and arithmetic of a Calabi-Yau Siegel threefold,
International Journal of Math. {\bf 29} (2011), 1561--1583.

%\bibitem [Ca]{Ca} \emph{J. W. Cassels}, Rational quadratic forms, Dover Pubn Inc (2008).

%\bibitem[CS]{CS} \emph{J. H. Conway and H. J. Sloane}, Sphere packings, lattices and groups. Third edition.
%Grundlehren der Mathematischen Wissenschaften, {\bf 290},
%Springer-Verlag, New York (1999).

%\bibitem[EZ]{EZ} \emph{M. Eichler and D. Zagier}, The Theory of Jacobi
%  Forms, Progress in Math. {\bf 55}, Birkh\"auser (1985).

%\bibitem[Fi]{Fi} {\em J. Fischer}, An approach to the Selberg trace
%  formula via the Selberg zeta-function, Lecture Notes in Mathematics
%  {\bf 1253}, Springer-Verlag (1987).

\bibitem[FrSa]{FS} {\em E. Freitag and R. Salvati Manni},
Some Siegel threefolds with a Calabi-Yau model II, preprint (2010).
%Kyungpook Mathematical Journal

\bibitem[GeNy]{GeNy} {\em B. van Geemen and N. O. Nygaard}, On the
  geometry and arithmetic of some Siegel modular threefolds, J. Number
  Theory {\bf 53} (1995), 45--87.

\bibitem[GrNi]{GN} {\em V. Gritsenko and V. Nikulin}, Automorphic forms and Lorentzian Kac-Moody algebras. Part II, Intern. J. of Math. {\bf 9} (1998), 201--275.

%\bibitem[GKZ]{GKZ} \emph{B. Gross, W. Kohnen, and D. Zagier}, Heegner
%  points and derivatives of $L$-series. II.  Math. Ann.  {\bf 278}
%  (1987), 497--562.

%\bibitem[GZ]{GZ} {\em B. Gross and D. Zagier}, Heegner points and derivatives of
%L-series, Invent. Math. {\bf 84} (1986), 225--320.

%\bibitem[Gri]{Gr} \emph{P. A. Griffiths}, Introduction to algebraic
%  curves, Amer. Math. Soc., Providence, Rhode Island
%  (1989).

\bibitem[Kn]{Kn} \emph{A. Knapp}, Elliptic Curves, Princeton University Press (1992).


%\bibitem[Ko]{Ko} {\em M. Kontsevich}, Product formulas for modular forms on $O(2,n)$, S\'eminaire Bourbaki {\bf 821} (1996).

\bibitem[Ku1]{Ku:Duke}  \emph{S. Kudla},
Algebraic cycles on Shimura varieties of orthogonal type.  Duke
Math. J.  {\bf 86}  (1997),  no. 1, 39--78.

%\bibitem[Ku2]{Ku:Annals} \emph{S. Kudla},
%Central derivatives of Eisenstein series and height pairings. Ann.
%of Math. (2) {\bf 146} (1997), 545--646.

\bibitem[Ku2]{Ku:Integrals}
\emph{S. Kudla}, Integrals of Borcherds forms, Compositio Math.
\textbf{137} (2003), 293--349.

%\bibitem[Ku4]{Ku:MSRI} {\em S. Kudla}, Special cycles and derivatives of Eisenstein series,
%in {\em Heegner points and Rankin $L$-series}, Math. Sci. Res.
%Inst. Publ. {\bf 49}, Cambridge University Press, Cambridge
%(2004).

%\bibitem[KM1]{KM1} {\em S. Kudla and J. Millson}, The theta correspondence and harmonic forms I, Math. Ann. {\bf 274}, (1986), 353--378.

%\bibitem[KM2]{KM2} {\em S. Kudla and J. Millson}, The theta correspondence and harmonic forms II, Math. Ann. {\bf 277}, (1987), 267--314.

\bibitem[KM]{KM3} {\em S. Kudla and J. Millson}, Intersection numbers of cycles on locally symmetric spaces and Fourier coefficients of holomorphic modular forms in several complex variables, IHES Publi. Math. {\bf 71} (1990), 121--172.

\bibitem[Ni]{Ni} \emph{V. V. Nikulin}, Integer symmetric bilinear forms and some of their geometric applications. (Russian) Izv. Akad. Nauk SSSR Ser. Mat. {\bf 43} (1979), 111--177, 238.
English translation in Mathematics of the U.S.S.R., Izvestia {\bf
14} (1980), 103--167.

%\bibitem[No]{No} A. Nobs, \emph{Die irreduziblen Darstellungen der Gruppen $\Sl_2(\Z_p)$, insbesondere $\Sl_2(\Z_2)$}. I. Teil, Comment Math. Helvetici {\bf 51} (1976), 465--489.

\bibitem[Sch1]{Sch:Inv}
 \emph{N. R. Scheithauer},
On the classification of automorphic products and generalized Kac-Moody algebras, Invent. Math. {\bf 164} (2006), 641--678.

\bibitem[Sch2]{Sch}
 \emph{N. R. Scheithauer}, Some constructions of modular forms for the Weil representation of $\SL_2(\Z)$, preprint (2011).

%\bibitem[Sk1]{Sk1} \emph{N.-P. Skoruppa}, Developments in the theory
%  of Jacobi forms. In: Proceedings of the conference on automorphic
%  funtions and their applications, Chabarovsk (eds.: N. Kuznetsov and
%  V. Bykovsky), The USSR Academy of Science (1990), 167--185. (see
%  also MPI-preprint 89-40, Bonn (1989).)

%\bibitem[SZ]{SZ} \emph{N.-P. Skoruppa and D. Zagier}, Jacobi forms
%and a certain space of modular forms,  Invent. Math.  {\bf 94}
%(1988), 113--146.

%\bibitem[Sh]{Sh} {\em G. Shimura}, Introduction to the Arithmetic Theory of Automorphic Functions, Princeton University Press, Princeton (1971).

%\bibitem[Shin]{Shin} { \em T. Shintani},  On the construction of
%holomorphic cusp forms of half integral weight, Nagoya Math. J.
%\textbf{58} (1975), 83--126.

%\bibitem[SABK]{SABK} {\em C. Soul\'e, D. Abramovich, J.-F. Burnol, and J. Kramer}, Lectures on Arakelov Geometry, Cambridge Studies in Advanced Mathematics {\bf 33}, Cambridge University Press, Cambridge (1992).

\bibitem[We]{We}
\emph{A. Weil}, Sur certaines groupes d'operateurs unitaires,
Acta Math. \textbf{111} (1965) 143--211.

%\bibitem[We2]{We2}
%\emph{A. Weil}, Sur la formule de Siegel dans la th\'eorie des
%groupes classiques, Acta Math. \textbf{113} (1965) 1--87.



\end{thebibliography}
\end{document}